\documentclass[12pt, reqno]{amsart}
\usepackage{fullpage}
\usepackage{amssymb, amsmath}
\usepackage{mathrsfs}
\usepackage{enumerate}
\usepackage[colorlinks,breaklinks, allcolors=blue]{hyperref}
\usepackage{xcolor}
\usepackage{longtable}
\usepackage{relsize}
\usepackage{cleveref}
\usepackage{tikz-cd} 

\let\svthefootnote\thefootnote
\newcommand\freefootnote[1]{%
  \let\thefootnote\relax%
  \footnotetext{#1}%
  \let\thefootnote\svthefootnote%
}

\newcommand{\mute}[1] {}

\newtheorem{theorem}[equation]{Theorem}
\newtheorem{lemma}[equation]{Lemma}
\newtheorem{proposition}[equation]{Proposition}

\newtheorem{noTitle}[equation]{}

\theoremstyle{definition}
\newtheorem{definition}[equation]{Definition}
\newtheorem{remark}[equation]{Remark}

\newtheorem*{assertion*}{Assertion}
\newtheorem{Setting}[equation]{Setting}
\newtheorem{Notation}[equation]{Notation}

\numberwithin{equation}{section}

\newcommand{\C}{\mathbb C}
\newcommand{\Z}{\mathbb Z}

\newcommand{\Q}{\mathbb Q}
\newcommand{\R}{\mathbb R}

\newcommand{\PP}{\mathbb P}

\newcommand{\CP}{\mathbb{C}P}

\DeclareMathOperator{\Lie}{Lie}

\DeclareMathOperator{\diag}{diag}

\DeclareMathOperator{\Morse}{Morse}
\DeclareMathOperator{\iso}{iso}

\DeclareMathOperator{\pt}{pt}
\DeclareMathOperator{\sph}{sph}
\DeclareMathOperator{\secc}{sec}
\DeclareMathOperator{\gen}{gen}

\newcommand{\acts}{\circlearrowleft}

\newcommand{\eps}{\varepsilon}

\begin{document}

\title[Equivariant cohomological rigidity of semi-free Hamiltonian circle actions]{On the equivariant cohomological rigidity of semi-free Hamiltonian circle actions}

\author{Liat Kessler}
\address{L.\ Kessler\\ Department of Mathematics, Physics, and Computer Science, University of Haifa,
at Oranim, Tivon 36006, Israel}
\email{lkessler@math.haifa.ac.il}

\author{Nikolas Wardenski}
\address{N.\ Wardenski\\ Department of Mathematics, University of Haifa, Haifa 3498838, Israel}
\email{wardenski.math@gmail.com}

\begin{abstract}
We consider semi-free Hamiltonian $S^1$-manifolds of dimension six and establish when the equivariant cohomology and data on the fixed point set determine the isomorphism type. Gonzales listed conditions under which the isomorphism type of such spaces is 
determined by fixed point data. We pointed out in an earlier paper that this result as stated is erroneous, and proved a corrected version. However, that version relied on a certain distribution of fixed points that is not at all necessary.
In this paper, we replace the latter assumption with a
global assumption on equivariant cohomology that is necessary for an isomorphism. 
We also extend our result to the equivariant (non-symplectic) topological category.
The variation in the earlier paper was tailored to suit the requirements of Cho's application of Gonzales' statement to classify semi-free monotone, Hamiltonian $S^1$-manifolds of dimension six. In the current paper, we aim to give the definitive statement relating fixed point data and equivariant cohomology to the isomorphism type of a semi-free Hamiltonian $S^1$-manifold.
\end{abstract}

\maketitle

\setcounter{secnumdepth}{1}
\setcounter{tocdepth}{1}
\tableofcontents
\section{Introduction}\label{sec:intro}

An effective action of a torus $T=(S^1)^d$ on a symplectic manifold $(M,\omega)$ is \textbf{Hamiltonian} if it admits a \textbf{momentum map}: $\mu\colon M\to \R^d \cong (\Lie(S^1)^d)^{*}$ with $$d\mu(\cdot)(\xi)=-\omega(\bar{\xi},\cdot)$$ for every $\xi\in \Lie(S^1)^d$, where $\bar{\xi}$ is the fundamental vector field of the action corresponding to $\xi$. An \textbf{isomorphism} between Hamiltonian $T$-manifolds is an equivariant symplectomorphism that intertwines the momentum maps.\\

It follows from the definition that the set of critical points of $\mu$ coincides with the fixed point set $M^{T}$.
A driving question in Hamiltonian geometry is how much information on the action 
is encoded by data on the fixed points. 
Assume that the manifolds are compact and connected.
Delzant \cite{De88} and Karshon \cite{Ka99} showed that when either the symmetry is maximal, i.e., the \emph{complexity} ${\dim M}/{2}-\dim T$ is zero (Delzant), or the manifold is of dimension four and the Hamiltonian action is of $S^1$ (Karshon), the isomorphism type is determined by a combinatorial invariant that is determined by data on the fixed point set. In complexity zero, this is the image of the momentum map, which is the convex hull of the images of the fixed points; when $\dim M=4$ and $T=S^1$, this is the decorated graph associated with the action.

Another natural question is: To what extent does the equivariant cohomology (which is an algebraic invariant) determine the equivariant diffeotype? 
For example, compact toric manifolds (smooth toric varieties) are {\bf equivariantly cohomologically rigid}: two compact toric manifolds
whose equivariant cohomology rings are isomorphic as algebras are equivariantly diffeomorphic \cite[Theorem~4.1]{masuda} and \cite[Remark~2.5(1)]{Ma20}. 
This holds, in particular, for compact Hamiltonian $T$-manifolds of complexity zero. By \cite{HKT25}, equivariant cohomological rigidity holds for compact four-dimensional Hamiltonian $S^1$-manifolds as well.\\

In this paper, we address both questions for semi-free Hamiltonian $S^1$-manifolds of dimension six.
Recall that an $S^1$-action is \textit{semi-free} if the stabilizers of the action are connected. 
This is the first case to consider when we relax the assumption on the complexity of the torus action or on the dimension of the manifold. In the semi-free Hamiltonian setting, the orbit space $\mu^{-1}(t)/S^1$ at a regular value $t$ of the momentum map is a manifold. By \cite{MW74}, the restriction of $\omega$ to $\mu^{-1}(t)$ descends to a symplectic form $\omega_t$ in $M_t:=\mu^{-1}(t)/S^1$, making $(M_t,\omega_t)$ a symplectic manifold, natural in the sense that any isomorphism $f\colon M\to M$ naturally induces a symplectomorphism $f_t\colon M_t\to M_t$.\\

Gonzales \cite{Go11} claimed in \cite[Theorems  1.5 and 1.6]{Go11} that any two compact, connected semi-free Hamiltonian $S^1$-manifolds $M^1$ and $M^2$ are isomorphic, that is, equivariantly symplectomorphic, provided that they have the same \textit{(small) fixed point data} (as defined in  \cite[Definitions 1.2 and 1.3]{Go11}) and that a certain rigidity assumption on the reduced spaces holds. 
Gonzales' theorem was key in Cho's proof of \cite[Theorem 1.1]{Ch21.2} (which was also carried out in \cite{Ch19}, \cite{Ch21.1}), stating that any semi-free, monotone Hamiltonian $S^1$-manifold is equivariantly symplectomorphic to a Fano variety with a holomorphic $S^1$-action.\\
However, as we pointed out in \cite[Lemma 2.3, Lemma 2.4]{KW25},
%\cite[Theorems 1.5 and 1.6]{Go11} are incorrect and 
more assumptions are needed to conclude that $M^1$ and $M^2$ are isomorphic. 
In \cite[Theorem 1.10]{KW25} we proved a variation of \cite[Theorem 1.6]{Go11}, in which we further assume that the reduced spaces of dimension four of $M^1$ and $M^2$ are symplectic rational surfaces, and require that non-extremal fixed surfaces, if exist, only occur at one level of the momentum maps. We showed in \cite[Theorem 1.14]{KW25} that this variation is enough to for the application of Cho.\\

The counterexample to \cite[Theorems 1.5 and 1.6]{Go11} that we gave in \cite[Example 2.1]{KW25} shows that it is, in general, not possible to extend an isomorphism $f$ between Hamiltonian $S^1$-manifolds that is defined below a non-extremal critical level $\lambda$ to above $\lambda$. 
In that example, the only fixed point components at 
$\lambda$ were fixed surfaces $F_1\subset M^1, F_2\subset M^2$.
An extension of $f$ over $\lambda$ necessarily sends the fundamental class $[F_1]\in H_2(M^1_{\lambda})$ to the fundamental class $[F_2]\in H_2(M^2_{\lambda})$, and this turned out to be incompatible with the action of $f$ on the homology of a reduced space below level $\lambda$.\\

In this article, we establish when the existence of an isomorphism between the equivariant cohomologies of the manifolds that is consistent with an isomorphism between the fixed point sets is sufficient for an isomorphism, both when isomorphism means equivariant symplectomorphism and when it means equivariant homeomorphism. 
In both results, we will drop the assumption on the distribution of the non-extremal fixed surfaces that we had in \cite[Theorem 1.10]{KW25}.\\

To deduce that an isomorphism below a critical level that is well-behaved cohomology-wise extends over the critical level, we need some assumptions on the manifolds.
In the main theorem's version in which 'isomorphism' means 'equivariant symplectomorphism intertwining the momentum maps', the manifolds will be in the following setting.
\begin{Setting}\label{set:symplectic}
Let $M$ be a compact, connected semi-free Hamiltonian $S^1$-space of dimension $6$, and $\mu$ a momentum map of $M$.
    Assume the following:
    \begin{itemize}
        \item Every fixed sphere of $M$ at a non-extremal critical value $\lambda$ is \emph{exceptional} in $M_{\lambda}$, i.e., it is an embedded symplectic sphere of self intersection $-1$.
        \item Every connected component of ${M}^{S^1}$ is simply-connected.
        \item Every reduced space of $M$ that is of dimension $4$ is a symplectic rational surface, i.e., a symplectic $S^2 \times S^2$ or $\CP^2 \# k \overline{\CP^2}$ with a symplectic blowup form.
        \item For all consecutive critical levels $\lambda<\lambda'$ and all $\lambda<t_0<t_1<\lambda'$, the pair $(M_{t_0},(\omega_t)_{t\in [t_0,t_1]})$ is rigid. 
    \end{itemize}
\end{Setting}

The assumption on the fixed spheres being exceptional is important because for those, being in the same homology class implies that they are ambiently symplectically isotopic, which does not hold for symplectic spheres in general. 
The second bullet means that every fixed surface is a sphere and that $M$ is simply-connected. 
The rigidity assumption is as in  \cite{Go11}; we give it in \Cref{sec:pre}. For more details on the notation of a symplectic rational surface, see, e.g., \cite[Notation 1.7]{KW25}.\\

In the theorem's version in which 'isomorphism' means 'equivariant homeomorphism', we can relax the requirement on $\mu$.
\begin{definition}\label{def:pseudomomentummap}
   Let $M$ be a connected, Hamiltonian $S^1$-space with proper momentum map $M\to \R$ whose image is bounded. 
   A \textbf{pseudo momentum map} on $M$ is a proper, $S^1$-invariant Morse-Bott function $\mu\colon M\to \R$  
   whose 
   critical set is the fixed point set of $M$,
   and that agrees with some momentum map of $M$ in a neighbourhood of each connected component of $M^{S^1}$.
\end{definition}

The example to have in mind 
is a slight perturbation of an actual momentum map $\tilde{\mu}$ near the fixed points, that is, $\mu=\tilde{\mu}+\rho$, where $\rho$ is an $S^1$-invariant smooth function that is supported and constant near the fixed point set.\\
We will extend notions regarding momentum maps to pseudo momentum maps, as $M_t=\mu^{-1}(t)/{S^1}$ or $f\colon M^1\to M^2$ being a $\mu-S^1$-diffeomorphism/$\mu-S^1$-homeomorphism, meaning that $f$ is an equivariant diffeomorphism/homeomorphism intertwining $\mu_1$ and $\mu_2$.

In this version, the manifolds will be in the next setting.
\begin{Setting}\label{set:nonsymplectic}
Let $M$ be a compact, connected semi-free Hamiltonian $S^1$-space of dimension $6$, and $\mu$ a pseudo momentum map on $M$.
    Assume the following:
    \begin{itemize}
        \item For any non-extremal critical value $\lambda$ and any fixed sphere $S\subset M_{\lambda}$, $M_{\lambda}\smallsetminus S$ has cyclic fundamental group.
        \item Every connected component of ${M}^{S^1}$ 
        is simply-connected.
    \end{itemize}
\end{Setting}
 The first bullet is needed to apply \cite[Theorem 6.1]{Su15}, which states that two embedded spheres in a compact, connected four-manifold are, under the condition that the fundamental group of their complement is the same cyclic group, ambiently isotopic if they are homotopic.\\

In both versions,
the requirement of ``same data on fixed points" consists of an isomorphism between the fixed point sets at critical levels that is consistent with a certain isomorphism between the equivariant cohomologies, as defined below, and an isomorphism between neighbourhoods of the minima. This requirement includes the ``same $*$-small fixed point data" that we had in \cite{KW25}.\\

Recall that the (Borel) {\bf equivariant cohomology} over a coefficient ring $R$, $R=\Z$ or $R=\Q$, of the action of $S^1$ on a manifold $M$ is
$$
H_{S^1}^*(M): = H^*((M\times ES^1)/{S^1}; R),
$$ 
where $ES^1$ is a contractible space on which $S^1$ acts freely, which can
be taken to be $S^{\infty}$. Here
 $S^{\infty}$ denotes the unit sphere in $\C^{\infty}$, and 
$S^1$ acts on $M\times S^{\infty}$ diagonally with orbit space $\C \PP^{\infty}$. 
The equivariant cohomology of a point is
\begin{equation*} \label{eqhs1pt}
H^*_{S^1}(\pt) =H^*(\C \PP^{\infty};\Z)  = \Z[t], \,\,\,\deg(t)=2.
\end{equation*}
 If $S^1$ acts freely on $M$, then $M\times_{S^1} S^{\infty}\cong M/S^1 \times S^{\infty}$, and so the equivariant cohomology of $M$ is the cohomology of the orbit space $M/S^1$ since $S^{\infty}$ is contractible. 
 The map $$\pi^*:H_{S^1}^*(\pt) \to H_{S^1}^*(M)$$ induced from the constant map $\pi:M\to \pt$ 
endows
$H_{S^1}^*(M)$ with a $H_{S^1}^*(\pt)$-algebra structure.

 \begin{definition}\label{def:muisomorphic}
    Let $\mu_1$ and $\mu_2$ be pseudo momentum maps (resp.\ momentum maps) on Hamiltonian $S^1$-manifolds $M^1$ and $M^2$ with the same set $\mathcal{C}$ of critical values.\\
    For a critical value $\lambda$, denote by $F_i$  
    the fixed point set of $M^i$ at level $\lambda$ with the orientation coming from the symplectic form. (It will be clear from the context to which $\lambda$ this set belongs.)
    Suppose that 
    \begin{itemize}
    \item for each critical level $\lambda$, there is an orientation-preserving homeomorphism (symplectomorphism) $\eta_{\lambda}\colon F_1\to F_2$, and 
    \item there are isomorphisms of algebras $\eta\colon H_{S^1}^*(M^2)\to H_{S^1}^*(M^1)$ and $\eta'\colon H_{S^1}^*((M^2)^{S^1})\to H_{S^1}^*((M^1)^{S^1})$ 
    such that the diagram with the horizontal maps $\eta$ and $\eta'$ and the vertical maps induced by inclusion
        \begin{equation*}
				\begin{tikzcd}
				H_{S^1}^*(M^2) \arrow{r} \arrow{d} &  H_{S^1}^*(M^1) \arrow{d} \\
				H_{S^1}^*((M^2)^{S^1}) \arrow{r}  & H_{S^1}^*((M^1)^{S^1}) 
			\end{tikzcd}
		\end{equation*}
        is commutative, and 
        %that
        \item $\eta'$ restricted to $H_{S^1}^*(F_2)$ agrees with the induced map $\eta_{\lambda}^*$ for all critical levels $\lambda$.
        \end{itemize}
        Then we say that $H_{S^1}^*(M^1)$ and $H_{S^1}^*(M^2)$ are \textbf{(symplectically) $\mu$-isomorphic} via $(\eta,\eta',(\eta_{\lambda})_{\lambda\in \mathcal{C} })$.
\end{definition}

\begin{theorem}\label{thm:mainresult}
    Let $(M^1,\mu_1)$ and $(M^2,\mu_2)$ be as in \Cref{set:nonsymplectic} (\Cref{set:symplectic}).
    Assume that 
     \begin{itemize}
        \item $H_{S^1}^*(M^1)$ and $H_{S^1}^*(M^2)$ are (symplectically) $\mu$-isomorphic via $(\eta,\eta',(\eta_{\lambda})_{\lambda\in \mathcal{C} })$.
        \item There is $\delta>0$ and an orientation-preserving, w.r.t.\ the symplectic orientations, homeomorphism (symplectomorphism) between $\mu_1^{-1}([\lambda_{min},\lambda_{min}+\delta])$ and $\mu_2^{-1}([\lambda_{min},\lambda_{min}+\delta])$, where $\lambda_{min}$ is the common minimal critical value.
       \end{itemize}
   Then there is an equivariant homeomorphism (equivariant symplectomorphism) $f\colon M^1\to M^2$ such that $f^*=\eta$ as isomorphisms $H_{S^1}^*(M^2)\to H_{S^1}^*(M^1)$.
\end{theorem}
    The last bullet is implied by the assumption that there is an orientation-preserving, w.r.t.\ the symplectic orientations, homeomorphism (symplectomorphism) solely between the minima of $M^1$ and $M^2$ that is covered by an isomorphism between their equivariant normal bundles. We will prove this in \Cref{lem:isomorphismneighborhoods}.\\
     
In the symplectic version, the assumption that the fixed spheres at non-extremal critical values are exceptional is new, compared to \cite[Theorem 1.10]{KW25}.
On the other hand, we drop the assumption that all non-extremal fixed surfaces, if exist, are mapped to the same critical value and this value is simple. \\

The key ingredient in the proof is the following theorem on extending an isomorphism over a critical level. We denote by 
$(M^i)^{\leq t}$ resp.\ $((M^i)^{S^1})^{\leq t}$ the set $\mu_i^{-1}((-\infty,t])$ resp.\ $\mu_i^{-1}((-\infty,t])\cap (M^i)^{S^1}$ for any $t\in \R$.
We use a similar notation for $\geq, <, >$, etc. We say that $\lambda-r$ is \emph{right below} the critical value $\lambda$ if there is no critical value in $[\lambda-r,\lambda)$ for $\mu_i$.
\begin{theorem}\label{thm:extendingnonsymplecticandsymplectic}
    Let $(M^1,\mu_1)$ and $(M^2,\mu_2)$ be as in \Cref{set:nonsymplectic} (\Cref{set:symplectic}). Assume that $H_{S^1}^*(M^1)$ and $H_{S^1}^*(M^2)$ are (symplectically) $\mu$-isomorphic via $(\eta,\eta',(\eta_{\lambda})_{\lambda\in \mathcal{C} })$.
    Let $\lambda$ be a common critical value of $\mu_1$ and $\mu_2$ and $r>0$ be such that $\lambda-r$ is right below $\lambda$ for $\mu_1$ and $\mu_2$.  
    Let $f\colon (M^1)^{\leq \lambda-r}\to (M^2)^{\leq \lambda-r}$ be an orientation-preserving equivariant homeomorphism (equivariant symplectomorphism) that intertwines $\mu_1$ and $\mu_2$ at level $\lambda-r$. If $\lambda$ is non-maximal, assume that the diagram
    \begin{equation*}
				\begin{tikzcd}
				H_{S^1}^*(M^2) \arrow[r, "\eta"] \arrow{d} &  H_{S^1}^*(M^1) \arrow{d} \\
				H_{S^1}^*((M^2)^{\leq \lambda-r}) \arrow[r, "f^*"]  & H_{S^1}^*((M^1)^{\leq \lambda-r})
			\end{tikzcd}
		\end{equation*}
        commutes.\\
        Then $f$ extends over the level $\lambda$ as an equivariant homeomorphism (equivariant symplectomorphism), meaning that there is $\delta>0$ and an equivariant, orientation-preserving homeomorphism (equivariant symplectomorphism)
 $$g\colon \mu_1^{-1}((-\infty,\lambda+\delta]) \to \mu_2^{-1}((-\infty,\lambda+\delta])$$
 such that $g=f$ on $\mu_1^{-1}((-\infty,\lambda-r])$ (on $\mu_1^{-1}((-\infty,\lambda-r-\delta])$). Also, it may be assumed that $g$ intertwines the pseudo momentum maps $\mu_1$ and $\mu_2$ near level $\lambda+\delta$, and if $\lambda$ is non-maximal that the diagram
 \begin{equation*}
				\begin{tikzcd}
				H_{S^1}^*(M^2) \arrow[r, "\eta"] \arrow{d} &  H_{S^1}^*(M^1) \arrow{d} \\
				H_{S^1}^*((M^2)^{\leq \lambda+\delta}) \arrow[r, "g^*"]  & H_{S^1}^*((M^1)^{\leq \lambda+\delta})
			\end{tikzcd}
		\end{equation*}
        commutes.
\end{theorem}
\begin{proof}[Proof of \Cref{thm:mainresult} assuming \Cref{thm:extendingnonsymplecticandsymplectic}]
We write 'isomorphism' for equivariant homeomorphism or equivariant symplectomorphism, according to the setting of $(M^i,\mu_i)$.
    By assumption, $(M^1,\mu_1)$ and $(M^2,\mu_2)$ have common critical values $\lambda_0<\hdots<\lambda_m$, and an isomorphism of neighborhoods of their minima.    
    So we can use \Cref{thm:extendingnonsymplecticandsymplectic} on that isomorphism to obtain an isomorphism $g\colon (M^1)^{\leq \lambda_1+\delta}\to (M^2)^{\leq \lambda_1+\delta}$. If $m\neq 1$, by \Cref{thm:extendingnonsymplecticandsymplectic} itself, the assumptions to apply \Cref{thm:extendingnonsymplecticandsymplectic} to $g$ are fulfilled, so we obtain an isomorphism $g\colon (M^1)^{\leq \lambda_2+\delta}\to (M^2)^{\leq \lambda_2+\delta}$. We can repeat the argument until we exhaust all of $M^1$ and $M^2$.
\end{proof}

The main task is therefore to prove \Cref{thm:extendingnonsymplecticandsymplectic}.\\
For non-extremal $\lambda$, let us start by pointing out the homological obstruction to extend an isomorphism over a critical level.
\begin{Notation}\label{not:canonical}
Let $M$ be a compact, connected semi-free Hamiltonian $S^1$-space of dimension $6$, oriented according to its symplectic form, and $\mu$ be a pseudo momentum map on $M$.\\
Let $\lambda$ be a non-extremal critical value.
 For 
 sufficiently small $r>0$,  
 there is an equivariant map $f_{\Morse}\colon \mu^{-1}(\lambda-r)\to \mu^{-1}(\lambda)$, called \textit{the Morse flow}, induced by the 'time-$r$-flow' of the normalized gradient vector field of $\mu$; see \S \ref{rem:criticalvalue}.

     Let $C$ be a connected component of the fixed point set $F$ of $M$ at $\lambda$ that is either a sphere or a point of index\footnote{We use the convention that the index of $\mu$ is half the Morse index.} $2$.
     The orbit space $C' \subset M_{\lambda-r}$ of the preimage of $C$ under the Morse flow $f_{\Morse}$ is a sphere. 
    Equip both $C'$ and $M_{\lambda-r}$ with the orientations coming from the symplectic form, and choose fundamental classes according to these orientations.
     Let $[C']\in H_2(M_{\lambda-r})$ be the image of the fundamental class of $C'$ under the embedding $C'\hookrightarrow M_{\lambda-r}$, and let $[C']^*\in H^2(M_{\lambda-r})$ be its Poincaré dual with respect to the fundamental class $[M_{\lambda-r}]$.
\end{Notation}
For $M^1$ and $M^2$, we get a collection of classes, one for each fixed sphere or fixed point of index $2$ at level $\lambda$. In fact, there is an even finer subdivision as follows.
We denote by:
\begin{enumerate}
    \item $\mathcal{D}^{\pt}_i$ the \textbf{multiset} (that is, allowing double counts\footnote{This is not necessary here, because these are classes of self-intersection $-1$ belonging to pairwise disjoint spheres.}) of primitive homology classes in $H_2(M^i_{\lambda-r})$ corresponding to those spheres in $M^i_{\lambda-r}$ that are mapped to a single point under $f_{\Morse}$.
    \item $\mathcal{D}^{\sph}_i(k,l)$  the \textbf{multiset} (that is, allowing double counts) of primitive homology classes in $H_2(M^i_{\lambda-r})$  corresponding to those spheres in $M^i_{\lambda-r}$ that are mapped to a fixed sphere under $f_{\Morse}$ whose negative resp.\ positive normal line bundle in $M^i$ is of degree $k$ resp.\ $l$.
\end{enumerate}
Certainly, if $f$ admits an extension over $\lambda$, then it has to respect this subdivision. In fact, we will show 
\begin{lemma} \label{lem:new}
In \Cref{thm:extendingnonsymplecticandsymplectic}, the map $f_{\lambda-r}$ sends $\mathcal{D}^{\pt}_1$ bijectively into $\mathcal{D}^{\pt}_{2}$ and $\mathcal{D}^{\sph}_1(k,l)$ bijectively into $\mathcal{D}^{\sph}_2(k,l)$.\\
    Also, for any sphere $C_1\subset F_1$, the equivariant normal bundles of $C_1$ and $\eta_{\lambda}(C_1)$ are isomorphic.
\end{lemma}
To prove the lemma, we will relate the classes in $\mathcal{D}^{\pt}_i$ and $\mathcal{D}^{\sph}_i(k,l)$ to the equivariant cohomology $H^*_{S^1}(M^i)$ as follows. To any such class $[C_i']$ in $H_{2}(M^{i}_{\lambda-r})$, we first consider its Poincaré dual $[C_i']^*$ in $H^2(M^i_{\lambda-r})$, view that as a class in $H^2_{S^1}(\mu_i^{-1}(\lambda-r))$, and then find a special extension of that class to $H^2_{S^1}(M^i)$. This extension enables us to use the isomorphism between the equivariant cohomologies of $M^1$ and $M^2$ to conclude that $f_{\lambda-r}$ indeed respects $\mathcal{D}^{\pt}_i$ and $\mathcal{D}^{\sph}_i(k,l)$.\\

Let us now define the type of class this extension will be.
Recall that the 
\emph{equivariant Euler class} $e^T(V)$ of an $S^1$-invariant vector bundle $V\to M$ is the Euler class of the vector bundle 
$V\times_{S^1} S^{\infty} \rightarrow M \times_{S^1} S^{\infty}$. The $i$-th 
equivariant Chern class $c_i^{S^1}(V)$ of a complex $S^1$-invariant 
vector bundle is defined similarly.
\begin{definition}\label{def:canonical}
    Let $C$ be a fixed point component of $M$ at level $\lambda$. A cohomology class $c \in H_{S^1}^{*}(M)$ is called {\bf canonical class of $C$} if $c$ vanishes on $M^{>\lambda}$, restricts to $C$ as the equivariant Euler class $e^+_{S_1}(C)$ of its positive normal bundle in $M$ (see \Cref{rem:criticalvalue}), and vanishes on all other fixed components at level $\lambda$.
\end{definition}
 By using the isomorphism $H_{S^1}^*(X;\Z)\to H^*(X/S^1;\Z)$ that exists for any free $S^1$-manifold $X$, we may evaluate a class $c\in H_{S^1}^*(X;\Z)$ on a homology class of $X/S^1$. This allows us to identify $H_{S^1}^*(\mu^{-1}({\lambda-r}))$ with $ H^*(M_{\lambda-r})$.

\begin{theorem}\label{thm:Poincare=Canonical}
Let $M$, $\mu$, $\lambda$, $r$, $[M_{\lambda-r}]$, $C$, $[C']$ and $[C']^{*}$ be as in \Cref{not:canonical}.
There exists a unique $c \in H_{S^1}^{*}(M)$ that is canonical class of $C$.
Moreover, the class $c$ restricts to $[C']^*\in H^2(M_{\lambda-r})$ under the natural $H^*_{S^1}(M)\to H_{S_1}^*(\mu^{-1}({\lambda-r}))\cong H^*(M_{\lambda-r})$ induced by the inclusion, that is, for a class in $H_{2}(M_{\lambda-r})$ represented by an immersed sphere\footnote{Any homology class in $H_2(M_{\lambda-r})$ can be represented by a smoothly immersed sphere $S\subset M_{\lambda-r}$.} $S$,
\begin{equation}\label{eq:toShow}
        c([S])=([C']^*\cup [S]^*)([M_{\lambda-r}]),
    \end{equation}
    where 
    $[S]^*$ is the Poincaré dual of $[S]\in H_2(M_{\lambda-r})$ w.r.t.\ the fundamental class $[M_{\lambda-r}]$.
\end{theorem}
The proof of the first part relies on the following result of Tolman and Weitsman.

\begin{proposition}\cite[Proposition 2.1]{TW99} \label{prop:prop2.1}
Let $M$ be a compact, connected Hamiltonian $S^1$-manifold with momentum map $\mu$, and $\lambda$ a critical value; choose $r>0$ small enough such that $\lambda$ is the only critical value in $(\lambda-r,\lambda+r)$.
    If $C:=M^{S^1}\cap \mu^{-1}(\lambda)$  is connected, the long exact sequence in equivariant cohomology for the pair $(M^{>\lambda+r},M^{>\lambda-r})$ splits into a short exact sequence
    \begin{align*}
        0\to H_{S^1}^*(M^{>\lambda-r},M^{>\lambda+r};\Q) \to H_{S^1}^*(M^{>\lambda-r};\Q) \overset{k^*}{\to} H_{S^1}^*(M^{>\lambda+r};\Q)\to 0.
    \end{align*}
    Moreover, the restriction $H_{S^1}^*(M^{>\lambda-r};\Q)\to H_{S^1}^*(C;\Q)$ induces an isomorphism from the kernel of $k^*$ to those classes in $H_{S^1}^*(C;\Q)$ that are multiples of $e_{\lambda}$, the equivariant Euler class of the positive normal bundle of $C$ in $M$ (as defined in \Cref{rem:criticalvalue}).\\
    If the action is semi-free, this is also true over $\Z$.
\end{proposition}
\begin{remark}
    The way the rational version of \Cref{prop:prop2.1} is stated is not the way it is stated in \cite{TW99}. However, it is an immediate consequence of it, by applying \cite[Proposition 2.1]{TW99} to $-\mu$.\\
    The integer version follows as the rational version does, using the semi-freeness to apply \cite[Lemma 6.1]{TW03} for the justification that cupping with the Euler class is an injection in \cite[Diagram (2.5)]{TW99}. Also, in that diagram, the isomorphism corresponding to the bottom vertical arrow stems from the equivariant version of the Thom isomorphism, which is simply the usual Thom isomorphism applied to the pullback of the negative normal bundle of $C$ under $C\times_{S^1} S^{\infty}\to C$.\\
    It turns out that the assumption on the semi-freeness is unnecessarily strong, but we don't need the stronger versions here.
\end{remark}

To show the existence and uniqueness of a canonical class $c$ of $C$, 
we need to first modify $\mu$ near the fixed point set, so that $C$ equals the fixed point set $F$ at some level of the modified map, as needed in order to apply \Cref{prop:prop2.1}. The modified map might no longer be a momentum map for the $S^1$-action, but is a pseudo momentum map.\\

This kind of modification is also important thereafter to use \cite[Theorem 7.1]{Su15}, which states that, under certain conditions, two homologous, topologically embedded spheres in a simply-connected four-manifold are ambiently isotopic. This theorem allows us to deduce from \Cref{lem:new} that the equivariant homeomorphism $f$  
in \Cref{thm:extendingnonsymplecticandsymplectic} 
can be extended below $\lambda$ such that it maps
$f_{\Morse}^{-1}(F_1)=F_1'$ to $f_{\Morse}^{-1}(F_2)=F_2'$ at a certain level right below the critical level $\lambda$. This would not always work if $F_1'$ resp.\ $F_2'$ had multiple connected components. Having achieved that, it turns out that we can simply use the Morse flow on both manifolds to extend $f$ beyond the critical level. 
In the symplectic version of \Cref{thm:extendingnonsymplecticandsymplectic}, \cite[Theorem 7.1]{Su15} is replaced by \cite[Lemma 5.20]{KW25}, which relies on results of Lalonde-Pinsonnault \cite{LP04}.
\begin{remark}\label{rem:momentummap}
    Many well-known theorems relating equivariant cohomology and the momentum map (as, for example, that $H_{S^1}^*((M^1)^{< t})\to H_{S^1}^*(((M^1)^{S^1})^{< t})$ is injective for all $t\in \R$ from \cite{Ki84}, but also \Cref{prop:prop2.1}) still hold true when $\mu_1$ and $\mu_2$ are pseudo momentum maps. This is quite clear from the method these statements are proven, namely by induction on the critical levels. Indeed, the actual properties needed from the  map to carry out the proofs of the mentioned theorems are the following.
    \begin{itemize}
        \item The maximal and minimal level sets are connected.
        \item The  map is an invariant Morse-Bott function with even Morse-Bott indices.
        \item The critical set equals the set of fixed points.
        \item Each component of the critical set is orientable.
    \end{itemize}
    These properties hold for pseudo momentum maps.
\end{remark}

The paper is organized as follows. In \Cref{sec:pre}, we recall the rigidity assumption and the definition and properties of the Morse flow.
In \Cref{sec:canonical}, we prove \Cref{thm:Poincare=Canonical} and some technical ingredients used for that. This includes a lemma that allows us to modify a momentum map into a pseudo momentum map that has exactly one fixed point component at each critical level.
In \Cref{sec:extend}, we first deduce \Cref{lem:new} from \Cref{thm:Poincare=Canonical} and then conclude \Cref{thm:extendingnonsymplecticandsymplectic}, extending an isomorphism below a critical level $\lambda$ beyond the critical level, both for non-extremal $\lambda$ and extremal $\lambda$.

\subsection*{Acknowledgements} 
We were motivated by Sue Tolman's assertion that assuming an isomorphism of equivariant cohomologies, in addition to same fixed point data, is required in order to prove that semi-free Hamiltonian $S^1$-manifolds are isomorphic. We thank Sue for helpful discussions. 
 The first author was supported in part by NSF-BSF Grant 2021730.

\section{Preliminaries}\label{sec:pre}

Let $(M,\omega,\mu)$ be a connected Hamiltonian $S^1$-manifold. Assume that the $S^1$-action is semi-free. Assume that the momentum map $\mu \colon M \to \R$ is proper, in the sense that a level set is compact, and has a bounded image.

\subsection*{Reduced spaces and the rigidity assumption.}
 For a regular value $t$ of $\mu$, the level set $P_t:=\mu^{-1}(t)$ is a compact manifold of dimension $\dim M-1$. It is connected because $\mu$ is Morse-Bott with even indices and $M$ is connected \cite{At82}.\\
Since the $S^1$-action is semi-free, it is free outside its fixed point set, so at a regular $t$, the {\em orbit space} $M_t:=P_t/S^1$ is a manifold 
   of dimension $\dim M-2$ and $S^1\to P_t\to M_t$ is a principal $S^1$-bundle. Moreover, the restriction of $\omega$ to $P_t$ descends to a symplectic form $\omega_t$ on $M_t$ by \cite{MW74}; we call $(M_t,\omega_t)$ the \textit{reduced space} at $t$.\\
    If $M$ is of dimension six  (and the action is semi-free), it turns out that even for non-extremal critical values $\lambda$, the orbit space $\mu^{-1}(\lambda)/S^1$, which we also call $M_{\lambda}$, can be given a smooth structure such that the symplectic form on $M$ descends to a symplectic form on the four-dimensional manifold $M_{\lambda}$.
    The case in which the fixed points at $\lambda$ are isolated is proven in \cite[Section 3.2]{Mc09}. The case in which there are also fixed surfaces 
    at $\lambda$ is in \cite[Section 3.3.1]{Go11}.\\
     If $\lambda$ is an extremal critical value, then $M_{\lambda}:=\mu^{-1}(\lambda)/S^1$ coincides with the fixed point set $F$ at level $\lambda$. The symplectic form $\omega_{\lambda}$ is then the restriction of the symplectic form on $M$ to $F$.\\
 We endow the manifold $M_t$ with the orientation induced by the symplectic form $\omega_t$, for $t$ regular or critical.

\begin{noTitle}\label{nt:morse1}
     Let $\lambda$ be a critical value of $\mu$. Let $\eps>0$ be such that there is no critical value in $[\lambda-\eps,\lambda)$. The normalized flow $\Phi_t$ of the gradient vector field of $\mu$ with respect to some invariant metric gives an equivariant diffeomorphism
     \begin{equation} \label{eq:phiti}
     \mu^{-1}(\lambda-\eps)\times [\lambda-\eps,\lambda)\to \mu^{-1}([\lambda-\eps,\lambda)), \quad (p,t)\mapsto \Phi_t(p)
     \end{equation}
     under which $\mu$ pulls back to $
     (p,t)\mapsto t.$ Choosing a different invariant metric does not change the equivariant isotopy type of this equivariant diffeomorphism.
\end{noTitle}

   Let $I=[t_0,t_1]\subset \mu(M)$ be an interval of regular values. Using the normalized gradient flow of $\mu$ w.r.t.\ some invariant metric, we obtain a smooth family of diffeomorphisms $M_{t_0}\cong M_t$, $t\in I$, and use this family to view all reduced forms $\omega_t$ on $M_t$ as symplectic forms on $M_{t_0}$.
    \begin{definition}\label{def:rigidityassumption}
        We say that $M$ satisfies the \textbf{rigidity assumption} if for all closed intervals $I=[t_0,t_1]\subset \R$ of regular values, $(M_{t_0},\{\omega_t\}_{t \in I})$ is rigid in the sense of \Cref{def:rigid} below.
    \end{definition}
     
     \begin{definition}\label{def:rigid}\cite[Definition 1.4]{Go11}.
        Let $B$ be a smooth manifold and $\{\omega_t\}$ a smooth family of symplectic forms on $B$, parametrized by real values $t\in I=[t_0,t_1]$ in a closed interval. We say that $(B,\{\omega_t\})$ is \textbf{rigid} if
        \begin{itemize}
            \item Symp$(B,\omega_{t})\cap \text{Diff}_{0}(B)$ is path-connected for all $t\in I$, where $\text{Diff}_0(B)$ is the identity component of the diffeomorphism group of $B$.
            \item Any deformation between any two cohomologous symplectic forms  that are symplectic deformation equivalent to $\omega_{t_0}$
             on $B$ may be homotoped through symplectic deformations with fixed endpoints into an isotopy, i.e., a symplectic deformation through cohomologous forms.
        \end{itemize}
    \end{definition}
\subsection*{The Morse flow}
 Assume further that $M$ is of dimension six.

\begin{noTitle}\label{rem:criticalvalue}
Let $\lambda$ be a non-extremal critical value of $\mu$.
We extend the $S^1$-action near any fixed surface $C$ to an effective, symplectic $T^2$-action that fixes $C$. This can be easily done by letting an additional circle act fiberwise on the normal bundle of $C$ in $M$.
We fix an $S^1$-invariant metric on $M$ which is standard Euclidean in the local normal form around each isolated fixed point, and $T^2$-invariant near all fixed spheres.
Then, we have a well-defined, equivariant and continuous map
    \begin{equation} \label{eq:morse}
         f_{\Morse}=f_{\Morse}(\eps)\colon \mu^{-1}({\lambda-\eps})\to \mu^{-1}({\lambda}), \quad f_{\Morse}(p)=\lim\limits_{\delta \to \eps_-} \Phi_{\delta}(p),
    \end{equation}
    where $\Phi_{\delta}$ is the time-$\delta$-flow with respect to the normalized gradient vector field of $\mu$.
    See \cite[3.5]{KW25}.
    We call this map the {\bf Morse flow}. Also, due to its equivariance, $f_{\Morse}$ descends to a map $M_{\lambda-\eps}\to M_{\lambda}$ between orbit spaces; we also call this the Morse flow.\\
    For a fixed point of index $1$ considered to be in $M_{\lambda}$, the preimage under $f_{\Morse}$ in $M_{\lambda-\eps}$ is a point. For a fixed point $p$ of index $2$, the preimage is an embedded symplectic 2-sphere $S_{-\eps}$ of size $\eps$ and self intersection $-1$. The preimage under the orbit map $M\to M/S^1$ of the union of these $S_{-\eps}$ (for $\eps$ running over $(0,\delta)$, $\delta>0$ sufficiently small) together with $p$ form a smooth $T^2$-invariant submanifold of dimension $4$; we call the normal bundle of $p$ in that submanifold the \textbf{negative normal bundle}. Similarly, we define the positive normal bundle if $p$ has index $1$.\\
    The spheres corresponding to different fixed points of index $2$ are pairwise disjoint.
    More concretely, the Morse flow $M_{\lambda-\eps}\to M_{\lambda}$ is a blowup map of $M_{\lambda}$ at the isolated fixed points of index $2$ in the topological category  \cite[Claim 3.13]{KW25}.\\
    
    The preimage of a fixed surface $\Sigma$ is an embedded symplectic surface $\Sigma_{-\eps}$ of the same genus which is fixed by the remaining $T^2/S^1=S^1$-action on a neighbourhood of $\Sigma$ in $M_{\lambda-\eps}$. The preimage under the orbit map $M\to M/S^1$ of the union of these $\Sigma_{-\eps}$ (for $\eps$ running over $(0,\delta)$, $\delta>0$ sufficiently small) form a smooth $T^2$-invariant submanifold of dimension $4$; we call the normal bundle of $\Sigma$ in that submanifold the \textbf{negative normal bundle}. Similarly, we define the \textbf{positive normal bundle}. We define $e({\Sigma})_-$ resp. $e(\Sigma)_+$ to be the Euler class of the negative resp. positive normal bundle of $\Sigma$; note that $e(\Sigma)_-$ pulls back to the restriction of the Euler class of the principal $S^1$-bundle $S^1\to \mu^{-1}({\lambda-\eps})\to M_{\lambda-\eps}$ to $\Sigma_{-\eps}$ under $\Sigma_{-\eps}\to \Sigma$.
    \end{noTitle}

\begin{Notation}\label{not:fixedpoints}
    We denote by $F$ the set of fixed points in $M_{\lambda}$, by $F_{\iso}$ its subset of isolated fixed points, by $F_{\iso, 2}$ its subset of isolated fixed points of index $2$, and by $F_{\sph}$ the union of fixed spheres.\\
    We denote by $F'$, $F'_{\iso}$, $F'_{\iso,2}$ and $F'_{\sph}$ the preimages of $F$, $F_{\iso}$, $F_{\iso,2}$ and $F_{\sph}$ in $M_{\lambda-\eps}$ under $f_{\Morse}$.\\

    Similarly, we define $F'$, etc., for preimages of $F$, etc., in $M_{\lambda+\eps}$. In case it is unclear if $F'$, etc., is considered to be in $M_{\lambda-\eps}$ or $M_{\lambda+\eps}$, we write $F'_{\lambda-\eps}$, $F'_{\lambda+\eps}$, etc.
\end{Notation}

 \section{The canonical class of a connected component of the fixed point set}\label{sec:canonical}

In this section we prove \Cref{thm:Poincare=Canonical}.
 We start by modifying the pseudo momentum map near the fixed point set $M^{S^1}$, so that a given connected component of $M^{S^1}$ equals the fixed point set at some level of the modified map.
We first look at the local model. 

\begin{noTitle}
Let $\Sigma$ be a point or a sphere. Denote by $\underline{\C^k_i}$, $i=\pm 1$,
the sum of $k$ complex line bundles over $\Sigma$ on which $S^1$ acts fiberwise as $t\cdot (z_1,\hdots,z_k)=(t^i\cdot z_1,\hdots, t^i\cdot z_k)$.
 Let $$V'=\underline{\C^{k_1}_{-1}}\oplus \underline{\C^{k_2}_1}$$ for some nonnegative integers $k_1$ and $k_2$, and let $n=k_1+k_2$ be the complex dimension of the fiber. For later applications, we will have $k_1=k_2=1$ if $\Sigma$ is a sphere, and either $k_1=1,k_2=2$ or $k_1=2,k_2=1$ if $\Sigma$ is a point.\\
 The fiberwise $S^1$-action extends to a certain fiberwise $T^n$-action that we fix from now on.
  By a theorem of Thurston \cite{Th76}, there is closed $2$-form on the bundle that restricts to the standard symplectic form on each fiber. By averaging, we can assume that it is $T^n$-invariant. Furthermore, the form is non-degenerate on a small enough invariant neighbourhood $V$ of the $0$-section; see details in \cite[5.3]{KW25}.   Then, the map 
  \begin{equation}\label{eq:s2}
  \mu\colon V\to \R, \; \mu(v)=\mu(v_1,v_2)=-|v_1|^2+|v_2|^2
  \end{equation}
  is a momentum map for the action.
 Here we denote by $|\cdot|$ the standard norm on the fiber of $\underline{\C^{k_1}_{-1}}$ resp. $\underline{\C^{k_2}_{1}}$.
\end{noTitle}

\begin{lemma}\label{lem:modifiedmomentmap}
There is $\delta>0$ such that for each $0<\eps<\delta$ (alternatively, $-\delta<\eps<0$), there is a smooth, $T^n$-invariant map $\mu'\colon V\to \R$ with the following properties.
    \begin{itemize}
        \item The critical set of $\mu'$ is precisely the $0$-section $V_0$ of $V$.
        \item $\mu'(V_0)=\eps$.
        \item There is a neighborhood $U$ of $V_0$ in $V$ such that $\mu'(v)=\mu(v)+\eps$ for all $v\in U$.
        \item There is a closed tubular neighborhood $U'$ of $V_0$ such that $\mu=\mu'$ on $V\smallsetminus U'$. For any real value $0<r$, the diameter of $U'$  can be chosen to be smaller than $r$ if $\delta>0$ as above is chosen to be small enough.
    \end{itemize}
\end{lemma}
 \begin{proof}
 We only need to show this for $0<\eps<\delta$. For a fixed $r>0$, let $\rho_r\colon [0,\infty)\to [0,1]$ be a smooth function such that $\rho_{r}$ is identically $1$ near $0$, identically $0$ on $[r,\infty)$, and decreasing. Define, for $\eps>0$,
     \[
     \mu_{\eps,r}'(v_1,v_2)=\mu(v_1,v_2)+\eps \rho_r(|v_1|^2+|v_2|^2)=(-|v_1|^2+|v_1|^2)+\eps\rho_r(|v_1|^2+|v_2|^2).
     \]
     It is clear that $\mu_{\eps,r}'=\mu+\eps$ near $V_0$ and $\mu_{\eps,r}'=\mu$ outside $U'=\{(v_1,v_2)\in V\colon |v_1|^2+|v_2|^2\leq r\}$. It is left to check that there is $\delta>0$ such that the critical set of $\mu'_{\eps,r}$ is $V_0$ for all $0<\eps<\delta$.\\

     Again, it is clear that this is always the case for the critical set of $\mu'_{\eps}$ in $U$, since there $\mu'=\mu+\eps$. Outside $U$, however, $|d\mu|$ is bounded from below by a constant $K>0$, and
     $d(\rho_{r}(|v_1|^2+|v_2|^2))$ is bounded from above by a constant $K'$ since $\rho_{r}$ has compact support. Hence, there is $\delta>0$ such that $K>\delta K'$, and we have
     $$|d\mu_{\eps,r}'|\geq |d\mu|-\eps |d(\rho_{r}(|v_1|^2+|v_2|^2))|\geq K-\eps K'\geq K-\delta K'>0$$
     outside $U$. This finishes the proof.
 \end{proof}

By the local normal form theorem for a semi-free Hamiltonian $S^1$-action on $(M^6,\omega)$ and by Weinstein's tubular neighbourhood theorem \cite{We71} (or \cite[Theorem 5.4]{KW25} for a concrete version), \Cref{lem:modifiedmomentmap} holds in a small enough $S^1$-invariant neighbourhood of an isolated fixed point or a fixed sphere.\\
 The next lemma implies that if the complement of a fixed component in a reduced space $M_{\lambda}=\mu^{-1}(\lambda)/S^1$ has cyclic fundamental group, then the same holds for a modification of $\mu$ as in \Cref{lem:modifiedmomentmap}.
 
 \begin{lemma}\label{lem:fundamentalgroup}
  Let $M$ be a semi-free Hamiltonian $S^1$-manifold of dimension $6$ and $\mu\colon M\to S^1$ a proper, pseudo momentum map. 
Let $\lambda$ be a non-extremal critical value.
Let $C$ be either an isolated fixed point or a fixed sphere in $\mu^{-1}(\lambda)$, and let $V$ be a tubular neighborhood of $C$ in $M$. Then, for $\eps>0$ (alternatively, $\eps<0$), $U$, $U'$ sufficiently small, there is $\mu'$ as in \Cref{lem:modifiedmomentmap} such that
     \begin{itemize}
         \item the complement of $C$ in $\mu^{-1}(\lambda)/S^1$ has the same fundamental group as the complement of $C$ in $(\mu')^{-1}(\lambda+\eps)/S^1$.         
         \item the complement of any fixed sphere $S$ in $(\mu')^{-1}(\lambda)/S^1$ has the same fundamental group as the complement of $S$ in $\mu^{-1}(\lambda)/S^1$.
     \end{itemize}
 \end{lemma}

 We write $M_{t}$ for the reduced space with respect to $\mu$, and $(\mu')^{-1}(t)/S^1$ for the reduced space with respect to $\mu'$. Let $U,U'$ be as in \Cref{lem:modifiedmomentmap}.
 \begin{proof}
 Without loss of generality, assume that $\eps>0$.
 The first statement is obvious if $C$ is a point, so assume $C$ is a sphere. Since the Morse flow $f_{\Morse}\colon M_{t}\to M_{\lambda}$ is a topological blow up  at the isolated fixed points of index $1$ for all $t$ right above $\lambda$, the complement of $C'=f_{\Morse}^{-1}(C)$ in $M_{t}$ has the same fundamental group as the complement of $C$ in $M_{\lambda}$. We thus only have to show that the complement of $C'$ in $M_{t}$ has the same fundamental group as the complement of $C$ in $(\mu')^{-1}(\lambda+\eps)/S^1$.\\
 
 If $U'$, $\eps$ are small enough, we can choose $t>\lambda+\eps$ in such a way that there is no $\mu$-critical value in $(\lambda,t]$ and $C'\subset M_t$ is contained in $(V\smallsetminus U')/S^1$. We claim that the Morse flow $f_{\Morse}'\colon \mu^{-1}(t)/S^1=(\mu')^{-1}(t)/S^1\to (\mu')^{-1}(\lambda+\eps)/S^1$ of $\mu'$ sends $C'$ into $C$. Indeed, the local fiberwise $T^2$-action on $V$, as in \Cref{rem:criticalvalue}, descends to a circle action on  a neighbourhood of $C'$ in  $M_t$ for which $C'$ is a  connected component of the fixed point set. Similarly, $C$ is the only connected component of the fixed point set w.r.t. the $T^2/S^1$-action in $(\mu^{-1}(\lambda)\cap V)/S^1$ and  the only connected component of the $T^2/S^1$-fixed set in $((\mu')^{-1}(\lambda+\eps)\cap V)/S^1$. It follows that $f_{\Morse}'$ maps $C'$ into $C$ due to equivariance of  $f_{\Morse}'$ with respect to this $T^2$-action.\\
 Thus, we obtain a homeomorphism $M_t\to (\mu')^{-1}(\lambda+\eps)/S^1$ sending $C'$ homeomorphically into $C$, which shows the claim.\\

 Now let $S$ be a fixed sphere in $(\mu')^{-1}(\lambda)/S^1$, so that in particular $S$ is a fixed sphere in $M_{\lambda}$. Then the fundamental group of the complement of $S$ in $M_{\lambda}$ is the same as the fundamental group of the complement of $S'={f_{\Morse}}^{-1}(S)\subset M_t$ in $M_t$, for all $t>\lambda$ such that there is no $\mu$-critical value in $(\lambda,t]$. The same holds for the preimage $(f_{\Morse}')^{-1}(S)$ of $S$ under the Morse flow $f_{\Morse}'\colon \mu^{-1}(t)/S^1=(\mu')^{-1}(t)/S^1\to (\mu')^{-1}(\lambda)/S^1$ of $\mu'$. \\
 
 Again, if we choose $t$ such that $S'\subset M_t$ is contained in $(V\smallsetminus U')/S^1$, the Morse flow $f_{\Morse}'\colon \mu^{-1}(t)/S^1=(\mu')^{-1}(t)/S^1\to (\mu')^{-1}(\lambda)/S^1$ sends $S'$ into $S$ due to local $T^2$-equivariance, so that $(f_{\Morse}')^{-1}(S)=S'$. This shows the last assertion and thus the whole lemma.
 \end{proof}
 In order to show the second part of \Cref{thm:Poincare=Canonical}, we must make use of the assumption that the canonical class $c$ vanishes on $M^{>\lambda}$; therefore, we need to have a way to relate  $M^{\lambda-r}$ and $M^{\lambda+r}$ for any $r>0$ such that only $\lambda$  is critical in $[\lambda-r,\lambda+r]$. This can be done using the Morse flow.\\
Indeed, by the definition of $f_{\Morse}$, it induces the homeomorphism
    $$f_{\Morse}\colon M_{\lambda-r}\smallsetminus F'_{\iso}\to M_{\lambda}\smallsetminus F_{\iso}$$
    and, similarly,  the homeomorphism 
    $$f_{\Morse}\colon M_{\lambda+r}\smallsetminus F'_{\iso}\to M_{\lambda}\smallsetminus F_{\iso}.$$
    These combine to a homeomorphism
    \begin{equation}\label{eq:homeo}
        h\colon M_{\lambda+r}\smallsetminus F_{\iso}'\to M_{\lambda-r}\smallsetminus F_{\iso}'.
    \end{equation}
    In the same fashion, there is a diffeomorphism
    \begin{equation}\label{eq:diffeo}
        g\colon M_{\lambda+r}\smallsetminus F'\to M_{\lambda-r}\smallsetminus F'
    \end{equation}
    that comes from the equivariant flow of the normalized gradient vector field in $M$; in particular, $c$ vanishes on $M_{\lambda-r}\smallsetminus F'$ because it does on $M_{\lambda+r}$.
    Therefore, for an immersed sphere contained in $M_{\lambda-r}\smallsetminus F'$, we already know that $c$ evaluates to $0$ on that sphere.
   Note that this statement does not hold for $M_{\lambda-r}\smallsetminus F_{\iso}'$ because \cref{eq:homeo} does \textit{not} come from an equivariant flow on $M$.\\
    Of course, not every homology class of $H_2(M_{\lambda-r})$ can be represented by a sphere in $M_{\lambda-r}\smallsetminus F'$. So the main part of the proof will be to decompose a given class $[A]\in H_2(M_{\lambda-r})$, considered as a class in $H_2((\mu^{-1}([\lambda-r,\lambda+r])\setminus F)/S^1)$, as $[A_1]+[A_2]$, where $[A_2]$ can be represented as a sphere in $M_{\lambda+r}\smallsetminus F'$, and $[A_1]$ can be written as the sum of homology classes corresponding to spheres that are in (the orbit space of) a tubular neighborhood of $F$. That way, we can compute $c(A_1+A_2)=c(A_1)$ because we know to what $c$ restricts  on neighborhoods of $F$.
 \begin{remark}\label{rem:decompsehomr}
    For $\lambda\neq t\in (\lambda-\eps,\lambda+\eps)$, we denote by $U_t$ a closed tubular neighborhood of $(F')^{\iso}_{t}$. We have
       \begin{equation} \label{eq:mv}
    H_2(M_{t})\cong H_2(U_{t})\oplus H_2(M_{t}\smallsetminus (F')^{\iso}_t).
    \end{equation}
Indeed, there is the Mayer-Vietoris sequence
    \[
    \hdots \to H_2(\partial U_{t})\to H_2(U_{t})\oplus H_2(M_{t}\smallsetminus (F')^{\iso}_t)\to H_2(M_{t})
    \overset{\partial_1}{\to} H_1(\partial U_{t})\to \hdots
    \]
    Note that $\partial U_{t}$ is diffeomorphic to a union of 3-spheres, so $H_1(\partial U_{t})=H_2(\partial U_{t})=0$. Equation \eqref{eq:mv} follows.
 \end{remark}
\begin{proof}[Proof of \Cref{thm:Poincare=Canonical}]
Let $C$ be a connected component of the fixed point set of $M$ at  a non-extremal critical level $\lambda$ of $\mu$ that is either a sphere or a point of index $2$.
Let $r>0$ such that $\lambda$ is the only critical value in $(\lambda-r,\lambda+r)$. %In order to show existence and uniqueness of the canonical class,
In order to apply \cite[Proposition 2.1]{TW99}, stated in \Cref{prop:prop2.1}, we modify $\mu$ such that $C$ becomes the only connected component at its level.
We apply \Cref{lem:modifiedmomentmap} on a local model $V$ near $C$ to obtain for $\frac{r}{2}>\eps>0$ arbitrarily small a new map $\mu'$ such that 
\begin{itemize}
    \item the sets of critical points of $\mu$ and of $\mu'$ are equal;
    \item $\mu'=\mu$ on $\mu^{-1}((-\infty,\lambda-2\eps])$, $\mu^{-1}([\lambda+r,\infty))$, and $F\smallsetminus C$; and 
    \item $\mu'(C)=\mu(C)-\eps$.
\end{itemize}
The notations $M_t$, $F$, $F'$, etc. are still with respect to $\mu$.\\
By applying \Cref{prop:prop2.1}
%\Cref{prop:prop2.1}  
to $\mu'$, we obtain a unique class $c$ that vanishes on $(\mu')^{-1}((\lambda-\eps),\infty)$ and restricts on $C$ to the equivariant Euler class of the positive normal bundle. Therefore, $c$ vanishes on $M^{>\lambda}\cap M^{S^1}$, which implies that it vanishes on $M^{>\lambda}$ by Kirwan injectivity, as well as on $F\smallsetminus C$. This proves the existence of a canonical class $c$ of $C$.\\

Let us now show that
    \begin{equation}\label{eq:toShow2}
        c([S])=([C']^*\cup [S]^*)([M_{\lambda-r}]),
    \end{equation}
    where $S\subset M_{\lambda-r}$ is an immersed sphere and $[C']^*$ and $[S]^*$ are the Poincaré duals of $[C']$ and $[S]$ w.r.t. the fundamental class $[M_{\lambda-r}]$.\\
Denoting by $U_{\lambda-r}$ a closed tubular neighborhood of $(F')^{\iso}_{\lambda-r}$  in $M_{\lambda-r}$, we use \cref{eq:mv} to write $[S]\in H_2(M_{\lambda-r})$ as $\tilde{A}_1\oplus [\tilde{S}_2]$ with 
$\tilde{A}_1\in H_2(U_{\lambda-r})$ and $\tilde{S}_2\subset M_{\lambda-r}\smallsetminus (F')^{\iso}_{\lambda-r}$  an immeresed sphere. Since \cref{eq:toShow2} is linear in $[S]$, it is enough to check the equality for $\tilde{A}_1$ and for $[\tilde{S}_2]$ separately.\\
    
    For $0\neq \tilde{A}_1\in H_2(U_{\lambda-r})$, we know  that $\tilde{A}_1$ is a sum of homology classes represented by immersed spheres, each contained in a single connected component of $U_{\lambda-r}$. Therefore, because of linearity, we might as well assume that $\tilde{A}_1=[\tilde{C}']$ for $\tilde{C}'$ that is a connected component of $(F')^{\iso}_{\lambda-r}$. We call the corresponding fixed point component $\tilde{C}$. Now and later we will use the fundamental property that if any two classes $[B_1]$ and $[B_2]$ are represented by immersed spheres $B_1,B_2$ in $M_{\lambda-r}$ that are transverse to each other, then we have
    \begin{equation}\label{eq:intersectionnumber}
        ([B_1]^*\cup [B_2]^*)([M_{\lambda-r}])=B_1\cdot B_2,
    \end{equation}
    where $B_1\cdot B_2$ is the intersection number of $B_1$ and $B_2$ (see \cite[Section 3.1 and 3.2]{Sc05}):
    \begin{itemize}
        \item If $\tilde{C}' \neq C'$, we have $c(\tilde{C}')=0$ because $c$ vanishes on all fixed point components at $\lambda$ that do not equal $C$, and the preimage of $\tilde{C}'$ under the orbit map $\mu^{-1}(\lambda-r)\to M_{\lambda-r}$ is equivariantly homotopic to $\tilde{C}$ under the Morse flow. On the other hand, we have $C'\cdot \tilde{C}'=0$ because $C'$ and $\tilde{C}'$ are disjoint.
        \item If $\tilde{S}_1=C'$, then $C$ is a fixed point of index $2$ and $C'$ is a sphere of self intersection $-1$, and we apply \Cref{lem:special1}.
    \end{itemize} 
    Let $[S]:=[\tilde{S}_2]\neq 0$ be in $H_2(M_{\lambda-r}\smallsetminus (F')^{\iso}_{\lambda-r})$. We will consider the class $[S]$ to be in $H_2((\mu^{-1}([\lambda-r,\lambda+r])\setminus F)/S^1)$ under the inclusion $M_{\lambda-r}\hookrightarrow (\mu^{-1}([\lambda-r,\lambda+r])\setminus F)/S^1$. We will show that it is possible to decompose it as $[S_1]+[S_2]$, where
    \begin{itemize}
        \item $S_2$ is an immersed sphere in $M_{\lambda+r}$, and
        \item $[S_1]$ is a class that can be written as the image of a union of immersed spheres that are contained in a tubular neighborhood around $F$.
    \end{itemize}
    That way, we will have $c([S])=c([S_1])+c([S_2])=c([S_1])$, which will allow us to reduce the claim to \Cref{lem:localcomputation}.\\
   % In showing that such a decomposition exists, 
    %we denote by $(F')^{\sph}_{\lambda-r}$ (similar for $\lambda+r$) the preimage of all fixed spheres at level $\lambda-r$ under $f^{\lambda-r}_{\Morse}$. 
  We may assume that $S$ is transverse to $(F')^{\sph}_{\lambda-r}$  and that the intersection points are pairwise different (see \cite[Theorem 2.4]{MH76}). For each intersection point $p'$ of $S$ and a component $\tilde{C}(p')$ of $(F')^{\sph}_{\lambda-r}$,  we consider
    \begin{itemize}
        \item the point $p=f_{\Morse}(p')\in F$,
        \item the corresponding fixed sphere $C(p)$, and 
        \item the \textbf{orbit space} $U(p)$ of a closed equivariant local model around $p$, containing precisely $p'$ from all the intersection points.  The orbit space $U(p)$ is given by a ball inside $\C_{-1}\oplus \C_1\oplus \C_0$ around $0$, where the index at the $\C$-summands is according to the weight of the $S^1$-action. 
    \end{itemize}
    We may choose a parametrization of $S$ near $p'$, that is a smooth injective map $$f_p\colon D^2\to S$$ such that $\Phi_T$, the normalized gradient flow of time $T$, sends  $f_p(D^2)\setminus \{p'\}$ into $U(p)$ whenever $T\leq 2r$. That is because we can choose any parametrization $f_p$ first and restrict it to a sufficiently small subdisk of $D^2$, if necessary.\\
    From now on, for a parametrization $f\colon S^2\to S$ of $S$, we view $D^2$ as embedded in $S^2$ such that $f_p=f$ on $D^2$.\\
    We now modify $f\colon S^2\to S\subset (\mu^{-1}([\lambda-r,\lambda+r])\setminus F)/S^1$.
    We decompose $S^2$ as
    $$(S^2 \smallsetminus D^2) \cup \{q \in D^2\,|\, ||q||\leq 1/2\} \cup \{q \in D^2\,|\, ||q||\geq 1/2\},$$
    where $D^2$ is considered as a unit disk.
    We identify the annulus $ \{q \in D^2\,|\, ||q||\geq 1/2\}$ as $S^1\times [1/2,1]\cong S^1\times [0,4]$.
  Define $$g_p \colon D^2 \to \mu^{-1}([\lambda-r,\lambda+r])\setminus F)/S^1$$ and $$g \colon S^2\smallsetminus D^2 \to\mu^{-1}([\lambda-r,\lambda+r])\setminus F)/S^1$$ as follows:
  \begin{itemize}
\item For $q \in S^2\smallsetminus (D^2\smallsetminus \partial D^2)$, set $g(q)=f(q)$ (the bottom part of the black line in \cref{fig:visualizationhomotopy} on the right picture).
\item For $q \in D^2$ with $||q|| \leq 1/2$, set $g_{p}(q)=f_p(2q)$ (the top part of the black line in \cref{fig:visualizationhomotopy} on the right picture).
\item For $q=(z,t)\in S^1 \times [0,4] (\cong S^1\times [1/2,1])$ with $t \leq 1$, set $g_p(z,t)=\Phi_{2tr}(f_p(z,0))$, where $\Phi_{2tr}\colon M_{\lambda-r}\setminus F'_{\lambda-r}\to M_{\lambda-r+2tr}$ is the normalized gradient flow (the red dotted line in \cref{fig:visualizationhomotopy} on the right picture, going rightwards).
 \item For $q=(z,t)\in S^1\times [0,4]$ with $1\leq t\leq 2$, set $g_p(z,t)=H(z,t)$, where $$H\colon S^1\times [1,2]\to U(p)\cap M_{\lambda+r}$$ is a homotopy such that $H(z,1)=g_p(z,1)$ and $H(z,2)=pt.\in U(p)\cap (F'_{\sph})_{\lambda+r}$ (the blue dashed line in \cref{fig:visualizationhomotopy} on the right picture, going upwards). 
        \item For $q=(z,t)\in S^1\times [0,4]$ with $2\leq t\leq 3$, we simply reverse $H$, that is, we set $g_p(z,t)=g_p(z,4-t)$ (the blue dashed line in \cref{fig:visualizationhomotopy} on the right picture, going downwards).
        \item For a point $q=(z,t)\in S^1\times [0,4]$ with $3\leq t\leq 4$, we simply reverse the normalized gradient flow, that is, we set $g_p(z,t)=g_p(z,4-t)$ (the red dotted line in \cref{fig:visualizationhomotopy} on the right picture, going leftwards).
        \end{itemize}
   
    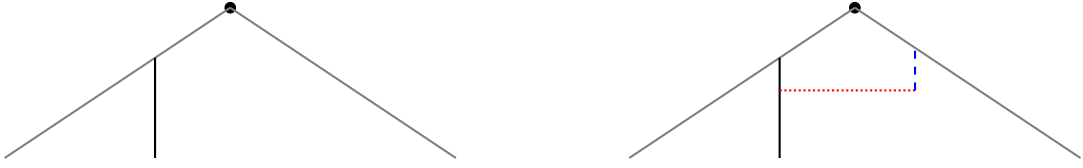
\begin{figure}
        \centering
        \hspace{-1cm}
        \begin{tikzpicture}
                \filldraw[black] (0,0) circle (2pt);
                \draw[gray, thick] (0,0) -- (3,-2);
                \draw[gray, thick] (0,0) -- (-3,-2);
                \draw[black, thick] (-1,-0.666) -- (-1,-2);
 \end{tikzpicture}
 \hspace{2cm}
 \begin{tikzpicture}
                \filldraw[black] (0,0) circle (2pt);
                \draw[gray, thick] (0,0) -- (3,-2);
                \draw[gray, thick] (0,0) -- (-3,-2);
                \draw[black, thick] (-1,-0.666) -- (-1,-1);
                \draw[black, thick] (-1,-1) -- (-1,-2);
                \draw[red, thick, densely dotted] (-1,-1.1) -- (0.8,-1.1);
                \draw[blue, thick, dashed] (0.8,-1.1) -- (0.8,-0.5666);
 \end{tikzpicture}
        \caption{A visualization of the homotopy from $f$ to $g$, the image has to be understood as the momentum map of a Hamiltonian $T^2$-action near $C$ that fixes $C$. On the left is (a part of) the initial map $f$, on the right (a part of) the map $g$.}
        \label{fig:visualizationhomotopy}
    \end{figure}
    By construction, $g$ and $g_p$ agree on $\partial D^2$, and therefore piece together to a continuous map $g\colon S^2\to (\mu^{-1}([\lambda-r,\lambda+r])\setminus F)/S^1$. This map is homotopic to $f$, because $g_p$ is homotopic rel $\partial D^2$ to $f_p$ by precomposing $g_p$ with a homotopy of $D^2$ that collapses a certain annulus to a circle. In particular, we have $$c([g(S^2)])=c([f(S^2)]).$$
    Further, since $g_p$ maps the circle $t=2$ in the annulus $D^2\supset S^1\times [0,4]$ as above to a point, the map $g$ factors through the canonical collapse of a circle: $S^2\to S^2\vee S^2$. Denote by $\tilde{g}_p$ the map $S^2\vee \{pt.\}\to U(p)\smallsetminus C(p)$ (whose image corresponds to the top solid line, the red dotted line going rightwards and the blue dashed line going upwards in the right picture of \cref{fig:visualizationhomotopy}) and by $\tilde{g}$ the map $\{pt.\}\vee S^2\to \mu^{-1}([\lambda-r,\lambda+r])\setminus F)/S^1$ (whose image corresponds to the blue dashed line going downwards, the red dotted line going leftwards and the bottom solid line in the right picture of \cref{fig:visualizationhomotopy}).
    \\
    Then we have $$c([f(S^2)])=c([g(S^2)])=c([\tilde{g}_p(S^2)])+c([\tilde{g}(S^2)]).$$ We may repeat the same argument for $\tilde{g}$ and the other points in the set $P_{\secc}$ of intersection points of $\tilde{g}(S^2)$ and $(F')^{\sph}_{\lambda-r}$. This gives the equation
    \[
    c([f(S^2)])=c([\tilde{g}(S^2)]+\sum\limits_{p\in P_{\secc}} c([\tilde{g}_p(S^2)]),
    \]
    where now $\tilde{g}(S^2)$ does not intersect $(F')^{\sph}_{\lambda-r'}$  at all, for any $r'>0$. Hence, using the normalized gradient flow, this sphere can be homotoped into $\mu^{-1}((\lambda,\lambda+r))/S^1$ through maps into $(\mu^{-1}([\lambda-r,\lambda+r])\smallsetminus F)/S^1$. So $c([\tilde{g}(S^2)])=0$. Using \Cref{lem:localcomputation}, the sum $\sum\limits_{p\in P_{\secc}} c([\tilde{g}_p(S^2)])$ is equal to the intersection number of $S$ with $C'$. This completes the proof of the theorem.

\end{proof}

 It remains to calculate the two hands of \cref{eq:toShow2} in two special cases of the theorem. First, we establish conventions regarding orientation.
Denote by $S^3_{-1,1}\subset \C_{-1}\oplus \C_{1}$, respectively $S^3_{1,1}\subset \C_1\oplus \C_1$,
the unit sphere equipped with the anti-diagonal, resp. diagonal, $S^1$-action.
Here, the index at the $\C$-summands denotes the weight of the $S^1$-action.  The orbit spaces of the actions are naturally equipped with an orientation $\mathcal{O}$, namely in such a way that any lift $X_1,X_2$ of an oriented normal frame on the base to $S^3$ satisfies that $(X_1,X_2,\xi)$ is positively oriented on $S^3$. Here $\xi$ is the canonical fundamental vector field of the $S^1$-action. Equivalently, this is the orientation coming from the symplectic form on these spheres, when they are being seen as reduced spaces of $\C_{-1}\oplus \C_{1}$ resp. $\C_{1}\oplus \C_{1}$ with the standard symplectic form at level $-r<0$ resp. $r>0$. That way, the descension of the equivariant diffeomorphism $S_{-1,1}^3\to S^3_{1,1}$ to the orbit spaces preserves the orientations.

\begin{lemma}\label{lem:special1}
    Let $V=\C_{-1}\oplus \C_{-1}\oplus \C_1$ be an $S^1$-representation, and let $S^3\subset \C_{-1}\oplus \C_{-1}\oplus \{0\}$ be the unit $3$-sphere, endowed with its standard orientation. Further, let $c\in H_{S^1}^2(0)$ be the first equivariant Chern class of the positive line bundle at $0\in V$, that is, $c$ is the Euler class of the line bundle $S^{\infty}\times_{S^1} \C_1$.\\
    Endow $S^2=S^3/S^1$ with the orientation $\mathcal{O}$, and denote by $[S^2]$ the corresponding fundamental class. Then under the restriction map $H_{S^1}^2(0)=H_{S^1}^2(V)\to H_{S^1}^2(S^3)\cong H^2(S^2)$, the image of $c$ evaluated on $[S^2]$ equals $-1$.
\end{lemma}
\begin{proof}
    By naturality, $c$ restricts to the Euler class of the bundle $S^3\times_{S^1} \C_1$ for $S^3$ endowed with the conjugate of the diagonal action. We have an isomorphism of bundles $$\phi\colon S^3\times_{S^1} \C_1\to S_{\diag}^3\times_{S^1} \C_1,$$ where $S^3_{\diag}$ is the unit $3$-sphere in $\C_1 \oplus \C_1$, given by
    \[
    [(z_1,z_2),v]\mapsto [(\overline{z}_1,\overline{z}_2),v].
    \]
    Again, view $S^2_{\diag}=S^3_{\diag}/S^1$ as a reduced space of $\C_1\oplus \C_1$ and denote by $[S_{\diag}^2]$ the fundamental class corresponding to this orientation. Then, since the restriction of $\phi$ to the $0$-section maps $[S^2]$ into $[S^2_{\diag}]$, we only have to show that the Euler class $e_{\diag}$ of $S_{\diag}^3\times_{S^1} \C_1$ evaluates to $-1$ on $[S_{\diag}^2]$. This holds since the latter bundle is the tautological bundle due to the embedding
    \[
    [(z_1,z_2),v]\mapsto ([z_1:z_2],v\cdot (z_1,z_2))\in \C \PP^1\times \C^2.
    \]
\end{proof}

\begin{lemma}\label{lem:localcomputation}
Let 
$M$, $\mu$, $\lambda$, $r$ 
%$[M_{\lambda-r}]$, $C$, $[C']$ and $[C']^{*}$
be as in \Cref{not:canonical}, and let $C$ be a fixed spere at level $\lambda$.
   % In the situation of \cref{thm:Poincare=Canonical}, 
   Let $p$ be a point in $C$, and $U=U(p)\cong (\C_{-1}\oplus \C_1)/S^1\times \C_0$ be the orbit space of an $S^1$-invariant, symplectic local model around $p$ in $M$. Denote by $C'$ the preimage of $C$ under $f_{\Morse}\colon M_{\lambda-r}\to M_{\lambda}$. By abuse of notation, we also write $C$ and $C'$ for the intersections of those with $U$, and let $g\colon S^2\to U(p)$ be a continuous embedding of a sphere such that
    \begin{itemize}
        \item $g(S^2)$ and $C'$ intersect transversely and exactly once;
        \item there is a closed tubular neighborhood $U_S\cong D^2$ of the south pole of $S^2$ such that $g(U_S)\subset M_{\lambda-r}$, and $g$ is smooth on $U_S$;
        \item there is a closed tubular neighborhood $U_N\cong D^2$ of the north pole of $S^2$ such that $g(U_N)\subset M_{\lambda+r}$;
        \item on $S^2\setminus (U_S\cup U_N)\cong S^1\times [0,2r]$, $g$ is given by $g(z,t)=\Phi_{t}(g(z,0))$, where $\Phi_t$ is the normalized gradient flow defined on a subset of $M_{\lambda-r}$.
    \end{itemize}
    When endowing $S^2$ with the orientation such that the intersection number of $g(U_S)$ with $C'$ with respect to the symplectic orientation in $M_{\lambda-r}$ equals $1$, we have $c([g(S^2)])=1$.
\end{lemma}
\begin{proof}
    The intersection number of $S$ with $C'$ is $1$ if and only if the parametrization
    \[
    g\colon U_S\cong D^2\to S \subset \{(z_1,z_2)\colon |z_1|=|z_2|+r\}/S^1\times \C_0
    \]
    composed with the projection
    \[
    \{(z_1,z_2)\colon |z_1|=|z_2|+r\}/S^1\times \C_0\to \{(z_1,z_2)\colon |z_1|=|z_2|+r\}/S^1\cong \C
    \]
    preserves orientation; otherwise it is $-1$. The composed map, in turn, preserves orientation if and only if the parametrization
    \[
    g\colon U_S\cong D^2\to f_{\Morse}^{\lambda-r}(S) \subset \{(z_1,z_2)\colon |z_1|=|z_2|\}/S^1\times \C_0
    \]
    composed with the projection
    \[
    \{(z_1,z_2)\colon |z_1|=|z_2|\}/S^1\times \C_0\to \{(z_1,z_2)\colon |z_1|=|z_2|\}/S^1\cong \C
    \]
    preserves orientation. 
    The outward pointing radial vector field outside $0$ is sent to an outward pointing vector field. 
    Therefore the orientation is preserved if and only if $\partial D^2\to \C\smallsetminus \{0\}\cong S^1\times (0,\infty)\to S^1$ preserves the standard orientations. 
    So we have that $g\colon \partial U_S\cong \partial D^2\to \C\smallsetminus \{0\}\cong S^1\times (0,\infty)\to S^1$ preserves standard orientations if and only if 
    \begin{equation}\label{eq:intersectionnumber}
         \text{the intersection number of } S \text{ with } C' \text{ is } 1.
    \end{equation}
    Now, we have a decomposition of $U\setminus C$ into $V_1/S^1$ and $V_2/S^1$, where
    \[
    V_1:=(\{(z_1,z_2)\neq (0,0)\colon |z_1|\geq |z_2|\}\times \C_0 ),\quad 
    V_2:=(\{(z_1,z_2)\neq (0,0)\colon |z_1|\leq |z_2|\}\times \C_0 ).
    \]
    Both of these sets are contractible, and their intersection is given by
    $V_1/S^1\cap V_2/S^1=\{(z_1,z_2)\neq (0,0)\colon |z_1|=|z_2|\}/S^1\times \C_0\cong S^1.$
    As a generator of $H_1(V_1\cap V_2)$, we choose the image of the circle
    $$S_{\gen}:=\{1\}\times S^1\times \{0\}\subset \{(z_1,z_2)\neq (0,0)\colon |z_1|=|z_2|\}\times \C_0$$
    under the projection map $V_1\cap V_2\to (V_1\cap V_2)/S^1$
    endowed with the standard orientation. Its preimage under the boundary map of the Mayer-Vietoris sequence associated to the above decomposition gives us a choice of a generator $G$ of $H_2(U\setminus C)$. We note that, when $\C \PP^1$ is equipped with the standard orientation and the (well-defined!) embedding $\C \PP^1 \to U\setminus C$ given by
    \begin{equation}\label{eq:embeddingProjSpace}
        [z_1:z_2]\mapsto
    \begin{cases}
        (1,z_2/z_1) \in V_1 & |z_1|\geq |z_2|\\
        (\overline{z}_1/\overline{z}_2,1) \in V_2 & |z_1|\leq |z_2|
    \end{cases}
    \end{equation}
    is considered, the image of $[\C \PP^1]$ is precisely $G$. This is since the map $\partial D^2\to \C \PP^1\to U\setminus C$
    has image $S_{\gen}$ and preserves standard orientations.\\
    Likewise, the image of $g$ intersects $\{(z_1,z_2)\colon |z_1|=|z_2|\}/S^1\times \C_0$ precisely in a circle by construction, and  therefore its image is a generator of $H_2(U\setminus C)$. By \cref{eq:intersectionnumber}, this is our chosen generator $G$ if and only if the intersection number of $S\cap U$ with $C'$ is $1$. Therefore, we want to show that the equivariant Euler class of the bundle $E:=((\C_{-1}\oplus\C_1)\smallsetminus \{0\})\times_{S^1} S^1$ evaluated on $G$ is $1$.\\
    
    Denote by $E^{\diag}$ the bundle $((\C_1\oplus\C_1)\smallsetminus \{0\})\times_{S^1} S^1$. That way, $E$ is isomorphic to the bundle $E^{\diag}$, namely via the equivariant diffeomorphism $(\C_{-1}\oplus\C_1)\smallsetminus \{0\}\to (\C_1\oplus\C_1) \smallsetminus \{0\}, (z_1,z_2)\mapsto (\overline{z}_1,z_2)$. The Euler class of $E^{\diag}$ evaluated on $[\C \PP^1]$, the image of the standard fundamental class of $\C \PP^1$ under the standard diffeomorphism $\C \PP^1\to S^3_{\diag}/S^1\subset E^{\diag}$, is $-1$. We now have the commutative diagram
    \begin{equation*}
		\begin{tikzcd}
                \C \PP^1 \arrow{d} \arrow{r} & \C \PP^1 \arrow{d}\\
                E \arrow{r} & E^{\diag}
		\end{tikzcd}
    \end{equation*}
    where the left vertical map is the one in \cref{eq:embeddingProjSpace}, the bottom horizontal map is complex conjugation in the first coordinate, and the top horizontal map is given by $[z_1,z_2]\mapsto [z_1,\overline{z}_2]$. Hence, the Euler class $e$ of the bundle $S^3\to \C \PP^1$ evaluated on the preimage of $[\C \PP^1]$, which is $-[\C \PP^1]$ under the top horizontal map, is $-1$. It follows that $e(G)=1$, as claimed.
\end{proof}

\section{Extending an isomorphism over a critical level}\label{sec:extend}

In this section we prove \Cref{thm:extendingnonsymplecticandsymplectic}, extending an isomorphism below a critical level beyond the critical level.

\subsection*{Case I: assume that $\lambda$ is non-extremal.}
We start by showing that an isomorphism extends under the condition that 
it is the identity near each preimage $C_i'$ under the Morse map of a connected component $C_i$ of the fixed point set at the critical level.\\

For $i=1,2$, let $M^i$ be a connected semi-free Hamiltonian $S^1$-space of dimension $6$, endowed with orientation coming from the symplectic form, and $\mu_i$ a pseudo momentum map on $M^i$. 
Assume that $\lambda$ is a non-extremal critical value of both $\mu_1$ and $\mu_2$, and that each of the connected components of the fixed point set $F_i$ at $\lambda$ is simply-connected. Let $\lambda-r$ be \textbf{right below} $\lambda$, that is, there is no critical value in $[\lambda-r,\lambda)$. Then, by \Cref{rem:criticalvalue}, $F_i'$ is a disjoint union of isolated points and spheres.
\begin{Notation}\label{not:identity}
    Let $C_1$ be a fixed component of $M^1$ and $C_2$ be a fixed component of $M^2$, both at the non-extremal critical level $\lambda$. Assume that they have isomorphic equivariant normal bundles in $M^1$ and in $M^2$. Let $U_i$ be a neighborhood of $C_i$ in $M^i$ such that there is an equivariant homeomorphism $g_i\colon U_i\to U$ to a fixed space $U$. Let $$h\colon V_1\to V_2$$ be an equivariant homeomorphism between subsets $V_i\subset U_i$.\\
    Then, for any subset $A\subset V_1$, we say that $h$ is \textbf{the identity on/near} $A$ if the map
    $$g_1(V_1)\to g_2(V_2),\quad x\mapsto (g_2\circ h\circ g_1^{-1})(x)$$
    is the identity on/near $g_1(A)\subset g_1(V_1)$.\\
    We make a similar definition for maps between the orbit spaces.\\

    In particular, denoting by $C_1'\subset F_1' \subset M^1_{\lambda-r}$ the component corresponding to $C_1$ and by $\pi_1\colon \mu_1^{-1}({\lambda-r})\to M^1_{\lambda-r}$ the orbit map, if $\pi_1^{-1}(C_1')\subset U_1$ and $V_1$ is a neighborhood of $C_1'$ in $M^1_{\lambda-r}$, then we can talk about $h$ being the identity near/on $C_1'\subset V_1$, and also about $h$ being the identity on/near $\pi_1^{-1}(C_1')\subset \pi_1^{-1}(V_1)$.
\end{Notation}
\begin{remark}\label{rem:htoh}
    As explained in \cite[Remark 5.15]{KW25}, if $h$ as in \Cref{not:identity} is equivariant and its descension to the orbit spaces is the identity near $C_1'$, then $h$ is isotopic through equivariant homeomorphisms to a map that is the identity near $\pi_1^{-1}(C_1')$. The point is that, in a tubular neighborhood $U$ of $\pi_1^{-1}(C_1')$, $h$ and the identity differ only by a map $U/S^1\to S^1$, which is nullhomotopic because $U/S^1\cong C_1'$ is simply-connected.
\end{remark}

\begin{lemma}\label{lem:topologicalextension}
     Let $C_1$ in $M^1$ and $C_2$ in $M^2$ be either fixed points of index $2$ or fixed spheres at level $\lambda$ with isomorphic equivariant normal bundles. Let $r>0$ be small enough such that $\lambda-r$ is right below $\lambda$ and  
     $C_i'$ is in the orbit space of a tubular neighborhood of $C_i$. Let $$f^{\lambda-r}\colon \mu_1^{-1}(\lambda-r)\to \mu_2^{-1}(\lambda-r)$$ be an equivariant homeomorphism whose induced map $f_{\lambda-r}$ on orbit spaces sends $C_1'\subset F_1'$ into $C_2'\subset F_2'$, preserving their orientations given by the symplectic forms.

     Then there is an isotopy $h^s$ from $f^{\lambda-r}$ to an equivariant homeomorphism $\mu_1^{-1}(\lambda-r)\to \mu_2^{-1}(\lambda-r)$ that is the identity on $\pi_1^{-1}(C'_1)$ (see \Cref{not:identity}).\\
     Further, if $C_1$ and $C_2$ equal the critical sets at level $\lambda$, then $f^{\lambda-r}$ extends to an equivariant homeomorphism $g \colon \mu_1^{-1}([\lambda-r, \lambda+r])\to \mu_2^{-1}([\lambda-r, \lambda+r])$.
 \end{lemma}
 \begin{proof}

    Let $V$ be a closed tubular neighborhood of $C_1'$ in $M^1_{\lambda-r}$.  As in \Cref{not:identity}, we may interpret $f_{\lambda-r}$ as a map $V \to M^1_{\lambda-r}$. Since $f_{\lambda-r}$ maps $C_1'\cong S^2$ into itself, preserving orientation, the identity on $C_1'$ is isotopic to the restriction of $f_{\lambda-r}^{-1}$ to $C_1'$. Suppose we have extended this isotopy $h_s$ to the whole of $M^1_{\lambda-r}$ such that $h_0=\text{id}$. Then, a composition with $f_{\lambda-r}$ would yield an isotopy of $f_{\lambda-r}$ to a homeomorphism $M^1_{\lambda-r}\to M^1_{\lambda-r}$ that is the identity on $C_1'$. By lifting this isotopy, in turn, to $\mu_1^{-1}(\lambda-r)$ (interpret the isotopy as a map $\mu_1^{-1}(\lambda-r)\times [0,1]\to \mu_2^{-1}(\lambda-r)\times [0,1]$ \cite[Lemma B.1]{KW25}), we would then obtain  $$h^s\colon \mu_1^{-1}(\lambda-r)\to \mu_2^{-1}(\lambda-r)$$ such that $h^0=f^{\lambda-r}$ and $h^1$ is the identity near $\pi_{1}^{-1}(C_1')$, by \Cref{rem:htoh}. \\
    
     Now let us find the extension of $h_s$ to $M_{\lambda-r}^{1}$. We can certainly lift $h_s$ to the boundary of the normal disk bundle $D^2\to V\to C_1'$ of $C_1'$ in $M_{\lambda-r}$ horizontally, using some connection of the $S^1$-principal bundle $\partial D^2\cong S^1\to \partial V\to C_1'$, starting at the identity; we also call this $h_s$ by abuse of notation. Hence, writing any element $v\in V\smallsetminus C_1'$ uniquely as $t\cdot v^{\partial}$ for $t\in (0,1]$ and $v^{\partial}\in \partial V$, we define
     $$h_s\colon V\to V, \quad h_s(v)=
     \begin{cases}
         t\cdot h_s(v^{\partial}), & v\in V\smallsetminus C_1' \\
         h_s(v), & v\in C_1'
     \end{cases}.$$
     Finally, by Corollary 1.2 and the remark below its proof in \cite{EK69}
     (which, combined, states that a topological isotopy of a submanifold that can locally be extended can globally be extended with support arbitrarily close to that submanifold), the isotopy extends to an isotopy $M^1_{\lambda-r}\to M^1_{\lambda-r}$ with support arbitrarily close to $C_1'$.\\
     
     To show the further assertion, assume that $C_i=F_i$. By the identification
     $$\mu_i^{-1}([\lambda-r,\lambda-r/2])\cong \mu_i^{-1}(\lambda-r)\times [\lambda-r,\lambda-r/2]$$
     obtained from the normalized gradient flow of $\mu_i$, we can consider $h^1$ as a map $\mu_{1}^{-1}(\lambda-r/2)\to \mu_{2}^{-1}(\lambda-r/2)$ and use the isotopy $h^s$ to
     extend $f^{\lambda-r}$ to a $\mu-S^1$-homeomorphism
     $$g\colon \mu_1^{-1}([\lambda-r,\lambda-r/2])\to \mu_2^{-1}([\lambda-r,\lambda-r/2])$$
     such that the map $g^{\lambda-r/2}$ it induces on  $\mu^{-1}(\lambda-r/2)$ equals $h^1$.
      Using that $g^{\lambda-r/2}$ is the identity on $\pi^{-1}(C_1')$, we can define an equivariant homeomorpism $g^{\lambda}\colon \mu_1^{-1}(\lambda)\to \mu_2^{-1}(\lambda)$ by setting $g^{\lambda}$ to be the identity on $C_1$, and
     $$g^{\lambda}(x)= (f^{\lambda-r/2}_{\Morse} \circ g^{\lambda-r/2}\circ (f^{\lambda-r/2}_{\Morse})^{-1})(x)$$
     otherwise.
     Finally, define the equivariant homeomorphism $g\colon \mu_1^{-1}([\lambda-r,\lambda+r])\to \mu_2^{-1}([\lambda-r,\lambda+r])$ by
     \[
     g(x)=
     \begin{cases}
        x & \text{if } f^{\mu_1(x)}_{\Morse}(x)\in C_1\\
        \left( \left(f^{\mu_1(x)}_{\Morse}\right)^{-1}  \circ g^{\lambda}\circ f^{\mu_1(x)}_{\Morse} \right)(x) &  \text{otherwise }
     \end{cases}.
     \]
     Here, $g(x)=x$ is  understood as $g$ being the identity on $C_1'\subset M_1^t$, for all $t\in [\lambda-r,\lambda+r]$.
     The restriction of $g$ to $\mu_1^{-1}(\lambda-r/2)$ is the previously defined $g^{\lambda-r/2}$, so we found our desired extension. 
 \end{proof}

We proceed to prove \Cref{thm:extendingnonsymplecticandsymplectic}. Let $(S^1 \acts M^1,\mu_1)$,  $(S^1 \acts M^2,\mu_2)$, $\lambda$ and $\eta_{\lambda},\eta,\eta'$ be as in the setting of the theorem.
%For $i=1,2$, we denote by $F_i$  the fixed point set of $M^i$ at level $\lambda$.
Let $\lambda-r$ be right below $\lambda$ and
\begin{equation}\label{eq:f}
    f\colon (M^1)^{\leq \lambda-r}\to (M^2)^{\leq \lambda-r}
\end{equation}
be any orientation-preserving equivariant homeomorphism that intertwines $\mu_1$ and $\mu_2$ near level $\lambda-r$, such that the diagram
    \begin{equation*}
				\begin{tikzcd}
				H_{S^1}^*(M^2) \arrow[r, "\eta"] \arrow{d} &  H_{S^1}^*(M^1) \arrow{d} \\
				H_{S^1}^*((M^2)^{\leq \lambda-r}) \arrow[r, "f^*"]  & H_{S^1}^*((M^1)^{\leq \lambda-r})
			\end{tikzcd}
		\end{equation*}
        commutes. Denote by $f^{\lambda-r}$ the induced map $\mu_{1}^{-1}(\lambda-r) \to \mu_{2}^{-1}(\lambda-r)$ and by $f_{\lambda-r}$ the induced map on the reduced spaces.\\

We first deduce from \Cref{thm:Poincare=Canonical} that the equivariant homeomorphism $f\colon (M^1)^{\leq \lambda-r}\to (M^2)^{\leq \lambda-r}$ acts appropriately on homology.
Recall the notation of $\mathcal{D}^{\pt}_i$ and $\mathcal{D}^{\sph}_i(k,l)$, given in the Introduction.
\begin{lemma}\label{lem:positivelinebundles}
    Let $[C'_1]$ be either in $\mathcal{D}^{\pt}_1$ or in $\mathcal{D}^{\sph}_1(k,l)$, and let $C_1$ be the corresponding fixed point set at level $\lambda$. Let $C_2:=
    \eta_{\lambda}(C_1)$ and $c_1$ and $c_2$ the unique canonical classes of $C_1$ and of  $C_2$, as in \Cref{def:canonical}. Then $\eta(c_2)=c_1$ on $\mu_1^{-1}([\lambda-r,\infty))=(M^1)^{\geq \lambda-r}$. In particular, $C_1$ and $C_2$ have isomorphic positive normal bundles, and $f_{\lambda-r}$ sends $[C_1']$ to $[C_2']$.
\end{lemma}
\begin{proof}
    Since the canonical homomorphism $H_{S^1}^*((M^1)^{\geq \lambda-r})\to H_{S^1}^*(((M^1)^{S^1})^{\geq \lambda-r})$  is injective (\cite{Ki84} for rational coefficients, and \cite[Proposition 6.2]{TW03} for an integer version), it suffices to check that these classes agree on the respective fixed point sets.\\
    Both classes clearly vanish on the fixed point sets not at level $\lambda$. So let us look at the fixed point sets at level $\lambda$. There, $c_1$ restricts on $C_1$ to the equivariant Euler class of the positive normal bundle of $C_1$ and restricts on all other components to $0$, by definition. Also, $c_2$ vanishes on $\eta_{\lambda}(F_1\smallsetminus C_1)$. Since $\eta'$ agrees with $\eta_{\lambda}^{*}$, this implies that $\eta(c_2)$ vanishes on $F_1 \smallsetminus C_1$. It is left to check that $c_1$ and $\eta(c_2)$ agree on $C_1$.\\
    By \Cref{prop:prop2.1}, $\eta(c_2)$ has to be an integer multiple of $c_1$ after restriction to $C_1$;  due to symmetry, $(\eta)^{-1}(c_1)$ has to be an integer multiple of $c_2$ after restriction to $C_2$. Hence  $c_1=\pm \eta(c_2)$. However, the restrictions of $c_1$ and $\eta(c_2)$ on $\{pt.\}\times \C \PP^{\infty}\subset C_1\times \C \PP^{\infty}$ agree because the restriction of both $c_1$ and $c_2$ to $\{pt.\}\times \C \PP^{\infty}$ is just the Euler class of $S^{\infty}\times_{S^1} \C_1$  and $\eta$ restricts to the identity on $H^*(\{pt.\}\times \C \PP^{\infty})$.\\
\end{proof}
We can now prove \Cref{lem:new}.
\begin{proof}[Proof of \Cref{lem:new}]
Let $C_1$ be a fixed point component of $M^1$ at level $\lambda$, and $C_2:=\eta_{\lambda}(C_1)$. We already know that $C_1$ and $C_2$ have isomorphic positive normal bundles by \Cref{lem:positivelinebundles}, and we know that $f_{\lambda-r}$ sends $[C_1']$ to $[C_2']$. So, showing that $f_{\lambda-r}$ sends $\mathcal{D}^{\pt}_1$/$\mathcal{D}^{\sph}_1(k,l)$ bijectively into $\mathcal{D}^{\pt}_2$/$\mathcal{D}^{\sph}_2(k,l)$ amounts to showing that $C_1$ and $C_2$ have isomorphic negative normal bundles.\\

These are determined by their Euler classes $e_{-}(C_1)$ and $e_{-}(C_2)$, we argue first that the homology class $f_{\lambda-r}(C_1')$ equals $C_2'$. Indeed, the restriction of $\eta$ to $H_{S^1}^*((M^2)^{\leq \lambda-r})$ is $f^*$, the map induced by $f$ on equivariant cohomology, by assumption. Since $\eta(c_2)=c_1$, 
we have $f^*(c_2)=c_1$ after restriction to $\mu_2^{-1}(\lambda-r)$.  By \Cref{thm:Poincare=Canonical}, the class $c_i$ is the Poincaré dual of $C_i'$ in $M^i_{\lambda-r}$; by assumption, $f_{\lambda-r}$ preserves orientation. Therefore $f_{\lambda-r}(C_1')=C_2'$.\\
    Now, $f_{\lambda-r}$ intertwines the Euler classes $e^1$ and  $e^2$ of the principal $S^1$-bundles  $S^1\to \mu_{1}^{-1}(\lambda-r) \to M^1_{\lambda-r}$  and $S^1 \to \mu_{2}^{-1}(\lambda-r) \to M^2_{\lambda-r}$, implying that $e^1(C'_1)=e^2(C'_2)$. On the other hand, $e^1([C'_1])=e_-(C_1)([C_1])$ and $e^2([C'_2])=e_-(C_2)([C_2])$ by \Cref{rem:criticalvalue}.  This completes the proof.
\end{proof}
 
 \begin{proof}[Proof of \Cref{thm:extendingnonsymplecticandsymplectic} in case $\lambda$ is non-extremal]
 Consider $f\colon (M^1)^{\leq \lambda-r}\to (M^2)^{\leq \lambda-r}$.
By \Cref{lem:new}, 
$f_{\lambda-r}$ maps bijectively $$\mathcal{D}_1^{\sph}:=\bigcup\limits_{k,l\in \Z} \mathcal{D}_1^{\sph}(k,l)$$
 to
$$\mathcal{D}_2^{\sph}:=\bigcup\limits_{k,l\in \Z} \mathcal{D}_2^{\sph}(k,l).$$
This is enough 
for the symplectic version of \Cref{thm:extendingnonsymplecticandsymplectic}, 
since then we can use \cite[Theorem 1.9]{KW25}.

Now let us deal with the non-symplectic version.  Denote by $\mathcal{C}^i_{\sph}$ the collection of all fixed spheres at level $\lambda$, by $\mathcal{C}^i_{\pt,2}$ the collection of all isolated fixed points with index $2$, and by $\mathcal{C}^i_{\pt,1}$ the collection of all isolated fixed points with index $1$. Using \Cref{lem:modifiedmomentmap} for negative $\eps$ repeatedly on the fixed spheres in $\lambda$ and then on the isolated fixed points of index $2$, we obtain $S^1$-invariant Morse-Bott functions $\mu_i'\colon M^i\to \R$ with the following properties.
     \begin{itemize}
         \item There is $r>0$ such that $\mu_i=\mu_i'$ outside $\mu_i^{-1}([\lambda-r,\lambda+r])$.
         \item More precisely, there are tubular neighborhoods $U$ and $U'$ of $F_i$, with $U\subset U'$, such that $\mu_i'-\mu_i$ is constant on every connected component of $U$, and $\mu_i=\mu_i'$ outside $U'$.
         \item The critical set of $\mu_i'$ equals $(M^i)^{S^1}$, and each critical level of $\mu_i'$ in $[\lambda-r,\lambda)$ contains precisely one fixed component of the $S^1$-action. 
         \item There is $r'>0$ such that $(\mu_i')^{-1}([\lambda-3r',\lambda-2r'])$ contains precisely the elements in $\mathcal{C}^{i}_{\sph}$, $(\mu_i')^{-1}([\lambda-2r',\lambda-r'])$ contains precisely the elements in $\mathcal{C}^{i}_{\pt,2}$, and $(\mu_i')^{-1}(\lambda)$ contains precisely the elements in $\mathcal{C}^{i}_{\pt,1}$.
         \item For the given diffeomorphism $\eta_{\lambda}\colon F_1\to F_2$, we have $\mu_2'\circ \eta_{\lambda}=\mu_1'$.
     \end{itemize}
The assumptions on $\mu_i$ in \Cref{thm:extendingnonsymplecticandsymplectic} hold for $\mu_i'$. Indeed, we only need to check that for any critical level $\lambda'$ of $\mu'$ in $[\lambda-3r',\lambda-2r']$, the fundamental group of the complement of the fixed sphere in $(\mu')^{-1}(\lambda')/S^1$ is cyclic. This is guaranteed by \Cref{lem:fundamentalgroup}.\\
It remains to show that the theorem holds with $\mu_i'$ replacing $\mu_i$. To simplify notation, we rename $\mu_i'$ into $\mu_i$.\\

     We will obtain a $\mu-S^1$-homeomorphism $\mu_1^{-1}([\lambda-r,\lambda+r])\to \mu_2^{-1}([\lambda-r,\lambda+r])$ that agrees with $f$ on level $\mu_1^{-1}(\lambda-r)$ and agrees with $\eta_{\lambda}$ on $F_1$ by first extending $f$ over all elements in $\mathcal{C}^1_{\sph}$, then over all elements in $\mathcal{C}^i_{\pt,2}$, and lastly over all elements in $\mathcal{C}^i_{\pt,1}$. If one of these sets is empty, there is no need to argue how to extend over it; so we assume that none of these sets are empty.

    Assume that $\lambda'$ is a critical value of $\mu_1$ in $[\lambda-3r',\lambda-2r']$, and let $C_1$ be the corresponding fixed sphere. Denote by $C_2$ the image of $C_1$ under $f_{\lambda'}$. Assume that we have already extended $f$ to a $\mu-S^1$-homeomorphism
    $$f\colon \mu_1^{-1}([\lambda-r,\lambda'-\delta])\to \mu_2^{-1}([\lambda-r,\lambda'-\delta]),$$
    where $\delta$ is such that there is no critical level in $[\lambda'-\delta,\lambda')$.
    Then the map $f_{\lambda'-\delta}$ induced on the reduced space $M_{\lambda'-\delta}$ sends the class $[C'_1]\in \mathcal{D}^{\sph}_1(k,l)$ to $[C'_2]\in \mathcal{D}^{\sph}_2(k,l)$ by \Cref{lem:new}. 
    The complement of $C_2$ in $M^2_{\lambda'}$ has cyclic fundamental group, hence so does the complement of $C_2'$ in $M^2_{\lambda'-r'}$ because the Morse flow $M^2_{\lambda'-\delta}\to M^2_{\lambda'}$ is a blowup map in the topological category.
    Thus, it follows from \cite[Theorem 6.1]{Su15} that there is a topological isotopy from $f_{\lambda'-\delta}(C_1')$ to $C_2'$ which comes from an ambient isotopy, and hence an isotopy $g_s\colon M^1_{\lambda'-\delta}\to M^2_{\lambda'-\delta}$, $s\in [0,1]$, such that $g_0=f_{\lambda'-\delta}$ and $g_1(C'_1)=C'_2$.\\
    Using a connection of the principal bundle $S^1\to \mu_1^{-1}(\lambda'-\delta)\to M^1_{\lambda'-\delta}$, we can lift $g_s$ to a continuous family of equivariant homeomorphisms $g^s\colon \mu_1^{-1}(\lambda'-\delta)\to \mu_2^{-1}(\lambda'-\delta)$ such that $g^0=f^{\lambda'-\delta}$. 
    By identifying $\mu_i^{-1}([\lambda'-\delta,\lambda'-\delta/2])$ with $M^i_{\lambda'-\delta}\times [\lambda'-\delta,\lambda'-\delta/2]$ using the normalized gradient flow of $\mu_i$, we obtain a level-preserving homeomorphism
    $$g\colon \mu_1^{-1}([\lambda'-\delta,\lambda'-\delta/2]) \to \mu_2^{-1}([\lambda'-\delta,\lambda'-\delta/2])$$
    such that $g^{\lambda'-\delta}=f^{\lambda'-\delta}$ and $g_{\lambda'-\delta/2}(C_1')=C_2'$. Finally,  \Cref{lem:topologicalextension} (which we may use since $C_1$ and $C_2$ have isomorphic normal bundles by \Cref{lem:positivelinebundles}) gives $\eps>0$ and the desired extension
    $$g\colon \mu_1^{-1}([\lambda'-\delta,\lambda'+\eps]) \to \mu_2^{-1}([\lambda'-\delta,\lambda'+\eps])$$
    of $g$ (and therefore $f$) over the level $\lambda'$.\\
    
    Now denote by $\lambda'$ any critical value in $[\lambda-2r',\lambda-r']$. By design, the only critical sets in $\mu_i^{-1}([\lambda-2r',\lambda-r'])$ are exceptional spheres, so in particular the critical set $C_i$ in $\mu_i^{-1}(\lambda')$ consists precisely of an isolated fixed point of index $2$. We have that $M^i_{\lambda-2r'}$ is homeomorphic to $M^i_{\lambda-r'}$. Therefore, if we let again $C_1',C_2'\subset M^1_{\lambda'-\delta},M^2_{\lambda'-\delta}$ be the spheres that are being mapped to $C_1,C_2$ under the Morse flow, we have that $g_{\lambda'-\delta}([C_1'])=[C_2']$ because of \Cref{lem:positivelinebundles}, and the complements of $C_i'$ in $M^i_{\lambda'-\delta}$ clearly have trivial fundamental group. Therefore, the same arguments as before can be applied to extend the $\mu-S^1$-homeomorphism
    $$f\colon \mu_1^{-1}([\lambda-r,\lambda'-\delta])\to \mu_2^{-1}([\lambda-r,\lambda'-\delta])$$
    over $\lambda'$, repeatedly, giving us a $\mu-S^1$-homeomorphism
    $$g\colon \mu_1^{-1}([\lambda-r,\lambda-\delta])\to \mu_2^{-1}([\lambda-r,\lambda-\delta])$$
    for any $\delta>0$ sufficiently small.\\
    
     Lastly, it is possible to extend $g$ over the fixed points $C_1$ of index $1$, that are all located at the level $\lambda$, because $g_{\lambda-\delta}$ is clearly isotopic to a homeomorphism $M^1_{\lambda-\delta}\to M^2_{\lambda-\delta}$ that maps the finite sets $C_1'$ and $C_2'$ into each other. This finishes the proof.
 \end{proof}

\subsection*{Case II: assume that $\lambda$ is maximal.}\label{subsec:maximal}
We now prove \Cref{thm:extendingnonsymplecticandsymplectic} in case $\lambda$ is maximal. The symplectic version was already proven in \cite[Theorem 1.9]{KW25}, which states, among other things, that any isomorphism defined between two Hamiltonian $S^1$-manifolds $M^1$ and $M^2$ right below a maximal level $\lambda$ yields an isomorphism $M^1\to M^2$, provided that the maxima are both points or both spheres.\\

For the non-symplectic version, we need some results from \cite{KW25}, as well, that we now restate for our current setting for the convenience of the reader.
\begin{lemma}\label{lem:neighborhoodsmaxima}
Let $M^1$ and $M^2$ be compact, connected, semi-free Hamiltonian $S^1$-manifolds of dimension six and let $\lambda-r$ be right below the common maximal critical value $\lambda$. Assume that the maxima are simply-connected.
\begin{enumerate}
    \item If the maxima are of dimension $2$, the space of orientation-preserving homeomorphisms $M^1_{\lambda-r}\to M^2_{\lambda-r}$ preserving the Euler classes of the bundles $S^1\to \mu_i^{-1}(\lambda-r)\to M^i_{\lambda-r}$ is connected (\cite[Lemma 5.11(1)]{KW25}).
    \item If the maxima are of dimension $2$, then an equivariant homeomorphism $\mu_1^{-1}({\lambda-r})\to \mu_2^{-1}({\lambda-r})$ yields an equivariant symplectomorphism of neighborhoods of the maxima (\cite[Lemma 5.11(2)]{KW25}) after possibly rescaling the symplectic forms of $M^1$ and $M^2$.
    \item If the maxima of $M^1$ and $M^2$ are of dimension $4$ and there is an equivariant homeomorphism $\mu_1^{-1}({\lambda-r})\to \mu_2^{-1}({\lambda-r})$, then there is $\delta>0$ such that $\mu_1^{-1}([\lambda-2\delta,\lambda])$ and $\mu_2^{-1}([\lambda-2\delta,\lambda])$ are equivariantly homeomorphic (\cite[Lemma 5.6]{KW25}).
\end{enumerate}
\end{lemma}

Before we prove \Cref{thm:extendingnonsymplecticandsymplectic}, we apply \Cref{lem:neighborhoodsmaxima} to show that the second bullet in \Cref{thm:mainresult} is implied by a seemingly weaker condition.

\begin{lemma}\label{lem:isomorphismneighborhoods}
Let $M^1$ and $M^2$ be compact, connected, semi-free Hamiltonian $S^1$-manifolds of dimension six with minimal value $0$.
    If there is an orientation-preserving, w.r.t.\ the symplectic orientation, homeomorphism (symplectomorphism) solely between the minima of $M^1$ and $M^2$ that is covered by an isomorphism between their equivariant normal bundles, then there is an equivariant homeomorphism (symplectomorphism) between closed neighborhoods $\mu_i^{-1}([0,\delta])$ for some $\delta>0$.
\end{lemma}
We need one additional ingredient, namely \cite[5.3]{KW25}, which relies ultimately on \cite{Th76} and \cite{We71}.
It says that for $M^1$ (similarly, $M^2$) there is a disk bundle $D$ over the minimum, on which $S^1$ acts fiberwise diagonally, with a certain invariant symplectic form $\omega$ such that $(D,\omega)$ is equivariantly symplectomorphic to $\mu_1^{-1}([0,\delta])$ for some $\delta>0$. Moreover, the symplectic form $\omega$ may be chosen in such a way that it restricts to the standard symplectic form on each disk fiber.\\
From that, we gain two immediate statements that we need in the proof of \Cref{lem:isomorphismneighborhoods}.
\begin{itemize}
    \item[(i)] Similar to \cref{eq:s2}, fiber multiplication with a value $s\in (0,1]$ sends the momentum map level $t$, $t\in [0,\delta]$, into the level $s^2t$.
    \item[(ii)] If the minima are of dimension $4$, then $\mu_1^{-1}(t)$ is a principal $S^1$-bundle over the minimum $F_1$ and $\mu_1^{-1}(t)/S^1=M^1_{t}$ is homeomorphic to $F_1$; under that homeomorphism, the Euler class of the normal bundle of $F_1$ is mapped to the Euler class of the $S^1$-bundle over $M^1_t$.
\end{itemize}
Now to \Cref{lem:isomorphismneighborhoods}:
\begin{proof}
     In the symplectic version, this is an immediate consequence of Weinstein's tubular neighborhood theorem, see \cite{We71} (or \cite[Theorem 5.4]{KW25} for a concrete statement), so we can assume that we are in the non-symplectic setting.
     \begin{itemize}
    \item If the minima are points, there is nothing to show.
    \item If the minima are spheres, we can rescale the whole symplectic form on $M^1$, for example, to a form $\omega_1'=s\cdot \omega_1$, $s>1$, with momentum map $\mu_1'$, so that now the minima are symplectomorphic. This gives us an equivariant symplectomorphism $g$ between neighborhoods $(\mu_1')^{-1}([0,\delta])$, $\mu_2^{-1}([0,\delta])$ of the minima of $(M^1,\omega_1')$ and $(M^2,\omega_2)$ by (2) of \Cref{lem:neighborhoodsmaxima}, and we may model these neighborhoods to be the total space $(D,\omega)$ of a disk bundle over a sphere. Since $(\mu_1')^{-1}([0,\delta])=\mu_1^{-1}([0,\delta/s])$, $g$ is an equivariant diffeomorphism
    $$ \mu_1^{-1}([0,\delta/s])\to \mu_2^{-1}([0,\delta]),$$
    but not yet a $\mu-S^1$-diffeomorphism. However, postcomposing $g$ with the fiberwise rescaling
    $$\mu_2^{-1}([0,\delta])\to \mu_2^{-1}([0,\delta/s]), \quad v\to v/\sqrt{s}$$
    yields the desired $\mu-S^1$-diffeomorphism by the item (i) above.
    \item If the minima are of dimension $4$, then for each $\delta>0$ sufficiently small, we have that $M^1_{\delta}$ is homeomorphic to $M^2_{\delta}$, because they are homeomorphic to $F_1$ and $F_2$, respectively by item (ii) above, and $F_1$ and $F_2$ are homeomorphic by assumption. Since $F_1$ and $F_2$ have isomorphic normal bundles by assumption, the $S^1$-bundles over $M^1_{\delta}$ and $M^2_{\delta}$ are also isomorphic (again, by item (ii) above); hence we obtain an equivariant homeomorphism $\mu_1^{-1}(t)\to \mu_2^{-1}(t)$ and thus the desired isomorphism of neighborhoods by \Cref{lem:neighborhoodsmaxima}.
\end{itemize}
\end{proof}

\begin{proof}[Proof of \Cref{thm:extendingnonsymplecticandsymplectic} in case $\lambda$ is maximal]
    There are three cases: the fixed point sets $F_1$ and $F_2$ at $\lambda$ are of dimension $0$, $2$ or $4$. In all cases, we show that there is $\delta >0$ such that both $\mu_i^{-1}([\lambda-2\delta,\lambda])$ are $\mu-S^1$-diffeomorphic to the same local model. We may as well assume that $\lambda-2\delta<\lambda-r$, since if it is not, we can use the normalized gradient flow to extend $f\colon \mu_1^{-1}((-\infty,\lambda-r])\to \mu_2^{-1}((-\infty,\lambda-r])$ to an equivariant homeomorphism
    \[
    f\colon \mu_1^{-1}((-\infty,\lambda-\delta])\to \mu_2^{-1}((-\infty,\lambda-\delta]).
    \]
    \begin{itemize}
        \item If $F_1$ and $F_2$  are of dimension $0$, they are both isolated fixed points at which the $S^1$-isotropy representation has weights $-1,-1,-1$. Hence there is $\delta>0$ such that $\mu_i^{-1}([\lambda-2\delta,\lambda])$ is even equivariantly symplectomorphic to a closed disk $D$ in $\C_{-1}\oplus \C_{-1} \oplus \C_{-1}$ with the standard symplectic form and a momentum map $\mu$ such that $0$ is mapped to $\lambda$. So, in order to extend $f$ to an equivariant homeomorphism $M^1\to M^2$, we only have to extend $f\colon D\supset \mu^{-1}(\lambda-\delta)\to \mu^{-1}(\lambda-\delta) \subset D$ to an equivariant homeomorphism $f\colon \mu^{-1}([\lambda-\delta,\lambda])\to \mu^{-1}([\lambda-\delta,\lambda])$. Again, using the normalized gradient flow of $\mu$, we identify $\mu^{-1}([\lambda-\delta,\lambda))\cong \mu^{-1}({\lambda-\delta})\times [\lambda-\delta,\lambda)$ and define $f\colon \mu^{-1}([\lambda-\delta,\lambda])\to \mu^{-1}([\lambda-\delta,\lambda])$ as
        \[
        f(x)=
        \begin{cases}
            (f^{\lambda-\delta}(p),t), \quad &
            \text{if } x=(p,t)\in \mu^{-1}({\lambda-\delta})\times [\lambda-\delta,\lambda)\\
            x \quad & \text{if } x\in \mu^{-1}({\lambda})=\{x\}.
        \end{cases}
        \]
        \item If $F_1$ and $F_2$ are of dimension $2$, we know by item (2) of \Cref{lem:neighborhoodsmaxima} that neighborhoods of the maxima are equivariantly symplectomorphic after rescaling the symplectic forms, which implies that neighborhoods $\mu_i^{-1}([\lambda-2\delta,\lambda])$ are $\mu-S^1$-diffeomorphic. Therefore, extending $f$ is again the same as extending $f$ considered as an equivariant homeomorphism $f\colon \mu_1^{-1}(\lambda-2\delta)\to \mu_1^{-1}(\lambda-2\delta)$ to an equivariant homeomorphism $\mu_1^{-1}([\lambda-2\delta,\lambda])\to \mu_1^{-1}([\lambda-2\delta,\lambda])$. This is possible because $f\colon \mu_1^{-1}(\lambda-2\delta)\to \mu_1^{-1}(\lambda-2\delta)$ is isotopic through equivariant homeomorphisms to the identity by item (1) of \Cref{lem:neighborhoodsmaxima}.
        \item If $F_1$ and $F_2$ are of dimension $4$, we know by item (3) of \Cref{lem:neighborhoodsmaxima} that neighborhoods $\mu_i^{-1}([\lambda-2\delta,\lambda])$ of the maxima are equivariantly homeomorphic (necessarily to a disk bundle over $F_i$ on which the $S^1$-action acts fiberwise, if $\delta$ is small enough). Then, extending $f$ is again the same as extending $f\colon \mu_1^{-1}(\lambda-2\delta)\to \mu_1^{-1}(\lambda-2\delta)$ to an equivariant homeomorphism $\mu_1^{-1}([\lambda-2\delta,\lambda])\to \mu_1^{-1}([\lambda-2\delta,\lambda])$, which is easily done using scalar multiplication on the $D^2$-fiber.
    \end{itemize}
\end{proof}

\bibliographystyle{amsalpha}

\end{document}